\documentclass[pdftex]{amsart}
\usepackage{amsmath, amssymb, url}
\usepackage{type1cm}
\usepackage{tikz, tikz-cd}
\usetikzlibrary{matrix,arrows,shapes,calc}
\newtheorem{defi}{Definition}[section]
\newtheorem{thm}[defi]{Theorem}
\newtheorem{prop}[defi]{Proposition}
\newtheorem{lem}[defi]{Lemma}
\newtheorem{cor}[defi]{Corollary}
\newtheorem{eg}[defi]{Example}
\newtheorem{rem}[defi]{Remark}
\newtheorem{conj}[defi]{Conjecture}

\DeclareMathOperator{\ad}{ad}
\DeclareMathOperator{\add}{add}
\DeclareMathOperator{\Adj}{Adj}
\DeclareMathOperator{\Auteq}{Auteq}
\DeclareMathOperator{\codim}{codim}
\DeclareMathOperator{\Cone}{Cone}
\DeclareMathOperator{\D}{D^b}
\DeclareMathOperator{\depth}{depth}
\DeclareMathOperator{\End}{End}
\DeclareMathOperator{\Ext}{Ext}

\DeclareMathOperator{\GL}{GL}
\DeclareMathOperator{\gldim}{gldim}

\DeclareMathOperator{\Hom}{Hom}
\DeclareMathOperator{\IW}{T}
\DeclareMathOperator{\lhom}{\mathcal{H}\!\mathit{om}}
\DeclareMathOperator{\lext}{\mathcal{E}\!\mathit{xt}}
\newcommand{\id}{\mathrm{id}}

\DeclareMathOperator{\Ker}{Ker}
\DeclareMathOperator{\KN}{KN}

\DeclareMathOperator{\modu}{mod}
\DeclareMathOperator{\NKN}{nKN}
\DeclareMathOperator{\PP}{\mathbb{P}}

\DeclareMathOperator{\Qcoh}{Qcoh}
\DeclareMathOperator{\rank}{rank}
\DeclareMathOperator{\Rep}{\widetilde{\mathcal{R}}}
\DeclareMathOperator{\RHom}{RHom}

\DeclareMathOperator{\SL}{SL}
\DeclareMathOperator{\Spec}{Spec}
\DeclareMathOperator{\Stab}{Stab}
\DeclareMathOperator{\Sym}{Sym}
\DeclareMathOperator{\Tilt}{\mathcal{T}}
\newcommand{\stsh}{\mathcal{O}}

\title[NCCR of minimal nilpotent orbit closures of type A]{non-commutative crepant resolution of \\ minimal nilpotent orbit closures of type A  \\ and Mukai flops}
\author{wahei hara}
\address{Department of Mathematics, School of Science and Engineering, Waseda University, 3-4-1 Ohkubo, Shinjuku, Tokyo 169-8555, Japan}
\email{waheyhey@ruri.waseda.jp}
\date{}

\begin{document}

\begin{abstract}
In this article, we construct a non-commutative crepant resolution (=NCCR) of a minimal nilpotent orbit closure $\overline{B(1)}$ of type A, 
and study relations between an NCCR and crepant resolutions $Y$ and $Y^+$ of $\overline{B(1)}$.
More precisely, we show that the NCCR is isomorphic to the path algebra of the double Beilinson quiver with certain relations
and we reconstruct the crepant resolutions $Y$ and $Y^+$ of $\overline{B(1)}$ as moduli spaces of representations of the quiver.
We also study the Kawamata-Namikawa's derived equivalence between crepant resolutions $Y$ and $Y^+$ of $\overline{B(1)}$ in terms of an NCCR.
We also show that the P-twist on the derived category of $Y$ corresponds to a certain operation of the NCCR,
which we call \textit{multi-mutation}, and that a multi-mutation is a composition of Iyama-Wemyss's mutations.
\end{abstract}

\maketitle

\tableofcontents

\section{Introduction}
The aim of this article is to study \textit{non-commutative crepant resolutions (=NCCR)} of a minimal nilpotent orbit closure $\overline{B(1)}$ of type A.
The notion of NCCR was first introduced by Van den Bergh \cite{VdB04b} in relation to the study of the derived categories of algebraic varieties.
We can regard the concept of NCCR as a generalization of the notion of \textit{crepant resolution}. Van den Bergh introduced it with an expectation that all crepant resolutions, whether commutative or not, have equivalent derived categories.
This expectation is a special (and non-commutative) version of a more general conjecture that K-equivalence implies derived equivalence.
We note that the study of NCCR is also motivated by theoretical physics (see Introduction of \cite{Le12}).

An NCCR of a Gorenstein algebra $R$ is defined as an endomorphism ring $\End_R(M)$ of a (maximal) Cohen-Macaulay $R$-module $M$ such that $\End_R(M)$ is a (maximal) Cohen-Macaulay $R$-module and has finite global dimension (see Definition \ref{def NCCR} and Lemma \ref{def lem NCCR}).
In relation to NCCR,  it is natural to ask the following questions.
\begin{enumerate}
\item[(1)] Construct an NCCR of $R$ and characterize a module $M$ that gives the NCCR.
\item[(2)] Construct a derived equivalence between the NCCR and a (commutative) crepant resolution.
\item[(3)] Construct a (commutative) crepant resolution as a moduli space of modules over the NCCR.
\end{enumerate}
For example, in \cite{BLV10}, Buchweitz, Leuschke, and Van den Bergh studied about these problems for a determinantal variety.
In this article, we deal with the above problems for a minimal nilpotent orbit closure $\overline{B(1)}$ of type A.
We also study about the derived equivalences for Mukai flops from the point of view of NCCR.

\subsection{NCCR of minimal nilpotent orbit closures of type A}
Let $V = \mathbb{C}^N$ be a complex vector space of dimension $N \geq 2$. 
Let us consider a subset $B(1)$ of $\End_{\mathbb{C}}(V)$ that is given by
\[ B(1) := \{ X \in \End_{\mathbb{C}}(V) \mid X^2 = 0, \dim \Ker X = N - 1 \}. \]
This is 
%an orbit under the $\SL_N$-action on $\End(V)$ via the adjoint representation
%\[ \Adj : \SL_N \to \GL(\End_{\mathbb{C}}(V)), \]
%and hence is 
a minimal nilpotent orbit of type A.
It is well known that the closure $\overline{B(1)}$ of the orbit $B(1)$
%\[ \overline{B(1)} := \{ X \in \End_{\mathbb{C}}(V) \mid X^2 = 0, \dim \Ker X \geq N - 1 \} \]
is normal and has only symplectic singularities, 
and thus the affine coordinate ring $R$ of $\overline{B(1)}$ is normal and Gorenstein.
Since $\codim_{\overline{B(1)}} (\partial B(1)) \geq 2$, we have a $\mathbb{C}$-algebra isomorphism
$R \simeq H^0(B(1), \stsh_{B(1)})$.
%First, we gives an NCCR of the minimal nilpotent orbit closure $\overline{B(1)}$.
%The minimal orbit $B(1)$ is equal to the orbit $\SL_N \cdot X_0$ where
%\[ X_0 = \begin{pmatrix}
%0 & 0 & \cdots & 0 \\
%\vdots & \vdots & \ddots & \vdots \\
%0 & 0 & \cdots & 0 \\
%1 & 0 & \cdots & 0
%\end{pmatrix} \in \End_{\mathbb{C}}(V). \]
%The stabilizer subgroup $\Stab_{\SL_N}(X_0)$ is given by
%\[ \Stab_{\SL_N}(X_0) = \left\{  \left( \begin{array}{c|ccc|c}
%c & 0 & \cdots & 0 & 0 \\ \hline
% &  &  & & 0 \\
%\ast &  &  A & & \vdots \\
% & & & & 0 \\ \hline
%\ast &  & \ast &  & c
%\end{array} \right) \mid 
%\begin{array}{c}
%A \in \GL_{N-2}, \\
%c \in \mathbb{C}^{\times}, \\
%c^2 \cdot \det(A) = 1 
%\end{array}
%\right\} \]
%(see, Lemma \ref{lem 2-stab}).
Let $H$ be a subgroup of $\SL_N$ such that $\SL_N/H \simeq B(1)$.
It is easy to see that $H$ is isomorphic to a subgroup of $\SL_N$
\[ H \simeq \left\{  A = \left( \begin{array}{c|ccc|c}
c & 0 & \cdots & 0 & 0 \\ \hline
 &  &  & & 0 \\
\ast &  &  A' & & \vdots \\
 & & & & 0 \\ \hline
\ast &  & \ast &  & c
\end{array} \right) \mid 
\begin{array}{c}
A' \in \GL_{N-2}, \\
c \in \mathbb{C}^{\times}, \\
c^2 \cdot \det(A') = 1 
\end{array}
\right\}. \]
Let $\mathcal{M}_a$ be a homogeneous line bundle on $B(1)$ that corresponds to the character $H \ni A \mapsto c^{-a} \in \mathbb{C}^{\times}$ and we set $M_a := H^0(B(1), \mathcal{M}_a)$.
We prove that a direct sum of $R$-modules $(M_a)_a$ gives an NCCR of $R$.

\begin{thm}[see \ref{NCCR1} and \ref{thm1 lem}] \label{1st main thm}
%There exist homogeneous vector bundles $\mathcal{M}_a ~ (a \in \mathbb{Z})$ on $B(1)$ that satisfy the following:
%The following hold.
\begin{enumerate}
\item[(a)] $M_a$ is a Cohen-Macaulay $R$-module  for $-N+1 \leq a \leq N-1$.
\item[(b)] For $0 \leq k \leq N-1$, the $R$ module $\bigoplus_{a = -N + k + 1}^k M_a$ gives an NCCR $\End_R\left(\bigoplus_{a = -N + k + 1}^k M_a\right)$
%\[ \Lambda_k := \End_R\left(\bigoplus_{a = -N + k + 1}^k M_a\right) \]
of $R$.
%\item[(c)] Each $\mathcal{M}_a$ corresponds to a character of $\Stab_{\SL_N}(X_0)$ that is given by
%\[ \Stab_{\SL_N}(X_0) \ni \left( \begin{array}{c|ccc|c}
%c & 0 & \cdots & 0 & 0 \\ \hline
% &  &  & & 0 \\
%\ast &  &  A & & \vdots \\
% & & & & 0 \\ \hline
%\ast &  & \ast &  & c
%\end{array} \right) \longmapsto c^{-a} \in \mathbb{C}^{\times}  \]
%\item[(c)] There is an equivalence of categories
%$\D(Y) \simeq \D(\modu(\Lambda_k))$.
\end{enumerate}
\end{thm}
%Moreover, we identify the representation of the stabilizer group of $B(1)$ that gives $\mathcal{M}_a$, explicitly.
The proof of Theorem \ref{1st main thm} is based on the theory of \textit{tilting bundles} on the crepant resolutions $Y$ and $Y^+$.
We note that the two crepant resolutions $Y$ and $Y^+$ of $\overline{B(1)}$ are the total spaces of the cotangent bundles on $\mathbb{P}(V)$ and $\mathbb{P}(V^*)$, respectively.
Let $\pi : Y \to \mathbb{P}(V)$ and $\pi' : Y^+ \to \mathbb{P}(V^*)$ be the projections.
We show that, for all $k \in \mathbb{Z}$, the bundles
\[ \Tilt_k := \bigoplus_{a = -N+k+1}^k \pi^*\stsh_{\mathbb{P}(V)}(a) ~~ \text{and} ~~ \Tilt^+_k := \bigoplus_{a = -N+k+1}^k \pi'^*\stsh_{\mathbb{P}(V^*)}(a) \]
are tilting bundles on $Y$ and $Y^+$, respectively (Theorem \ref{NCCR1}).
We also show that there is a canonical isomorphism of $R$-algebras
\[ \Lambda_k := \End_Y(\Tilt_k) \simeq \End_{Y^+}(\Tilt^+_{N-k-1}) \]
and that this algebra is isomorphic to the one that appears in Theorem \ref{1st main thm} (b).
Moreover, by the theory of tilting bundles, we have an equivalence of categories $\D(Y) \simeq \D(\modu(\Lambda_k))$
between the derived category of a crepant resolution and of an NCCR.
In Section \ref{subsec Rem NCCR}, we provide another NCCR $\Lambda'$ of $R$ that is not isomorphic to $\Lambda_k$ but is derived equivalent to $\Lambda_k$.

\subsection{NCCR as the path algebra of a quiver}

Next, we describe an NCCR $\Lambda_k$ of $R$ as the path algebra of the \textit{double Beilinson quiver} with some relations.
%This is the second main theorem in this paper.
We note that similar results for non-commutative resolutions of determinantal varieties are obtained by Buchweitz, Leuschke, and Van den Bergh \cite{BLV10}, and Weyman and Zhao \cite{WZ12}.

Let $S = \Sym^{\bullet}(V \otimes V^*)$ be the symmetric algebra of a vector space $V \otimes_{\mathbb{C}} V^*$.
%This is the affine coordinate ring of $\End_{\mathbb{C}}(V) \simeq V^* \otimes_{\mathbb{C}} V$.
Let $v_1, \dots, v_N$ be the standard basis of $V = \mathbb{C}^N$ and $f_1, \dots, f_N$ the dual basis of $V^*$.
We regard $x_{ij} = v_j \otimes f_i \in S$ as the variables of
the affine coordinate ring of an affine variety $\End_{\mathbb{C}}(V) \simeq V^* \otimes_{\mathbb{C}} V$.
Since $\overline{B(1)}$ is a closed subvariety of $\End_{\mathbb{C}}(V)$, $R$ is a quotient of $S$.

\begin{thm}[= Thm. \ref{NCCR quiver}]
As an $S$-algebra,
the non-commutative algebra $\Lambda_k$ is isomorphic to the path algebra $S\widetilde{\Gamma}$ of the double Beilinson quiver $\widetilde{\Gamma}$
with $N$ vertices
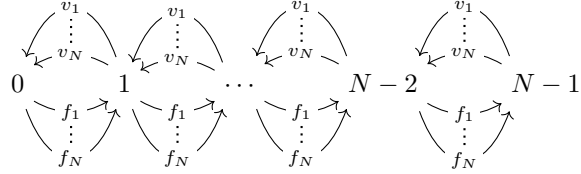
\[ \begin{tikzpicture}[
 description/.style={midway,
                     fill=white,
                     font=\scriptsize,
                     inner sep=2pt,
                     minimum height=.8em
 }]
 \matrix[matrix of math nodes, column sep=1cm]{ 
    |(n1)| 0 & |(n2)| 1 & |(dots)| \cdots & |(N1)| N - 2 & |(N)| N - 1 \\
 };
 \foreach \f/\t in {n1/n2, n2/dots, dots/N1, N1/N}{ 
  \begin{scope}
   \coordinate (ulc) at ($(\f.south east)-(3pt,3pt)$);
   
   \coordinate (urc) at ($(ulc -| \t.west)+(3pt,0)$);
   \path[->,every node/.style={description}] % 
    (\f.south east) edge[bend right=30] node (f1) {$f_1$} (\f.south east -| \t.west)
    (ulc) edge[bend right=70,looseness=2] node (fN) {$f_N$} (urc)
    ; 
   \path[draw=black, dotted, thick] (f1.south) to (fN.north);
   \end{scope}
  };
  \foreach \f/\t in {n2/n1, dots/n2, N1/dots, N/N1}{ 
   \begin{scope}
   \coordinate (alc) at ($(\f.north west)+(3pt,3pt)$);
   \coordinate (arc) at ($(alc -| \t.east)+(-3pt,0)$);
   \path[->,every node/.style={description}]
    (\f.north west) edge[bend right=30] node (vN) {$v_N$} (\f.north west -| \t.east)
    (alc) edge[bend right=70,looseness=2] node (v1) {$v_1$} (arc)
    ;
   \path[draw=black, dotted, thick] (vN.north) to (v1.south);
  \end{scope}
 };
\end{tikzpicture} \]
with relations
\begin{align*}
v_iv_j = v_jv_i, ~~ %~~ \text{for all $1 \leq i, j \leq N$}, \\
f_if_j &= f_jf_i, ~~ %~~ \text{for all $1 \leq i, j \leq N$}, \\
 v_jf_i = f_iv_j = x_{ij} ~~ \text{for all $1 \leq i, j \leq N$}, \\
\text{and} ~~ &\sum_{i=1}^N f_iv_i = 0 = \sum_{i=1}^N v_if_i.
\end{align*}
\end{thm}

Building on this theorem, we can also show that the two crepant resolutions $Y$ and $Y^+$ are recovered from the quiver $\widetilde{\Gamma}$ as moduli spaces of representations (Theorem \ref{NCCR to CR}).
The idea of the proof is based on the fact that crepant resolutions $Y$ and $Y^+$ are moduli spaces that parametrizes representations of Nakajima's quiver of type $\mathrm{A}_1$.
%This is the third main theorem in this paper.
%\begin{thm}[= \ref{NCCR to CR}] \label{3rd main thm}
%Let $\widetilde{\Gamma}_0 := \{0, 1, \dots, N-1\}$.
%We have the following.
%\begin{enumerate}
%\item[(a)] The crepant resolution $Y$ is the fine moduli space for the $\widetilde{\Gamma}$-representations $W = (W_i)_{i \in \widetilde{\Gamma}_0}$ of dimension vector $(1,1, \dots, 1)$ that is generated by the first component $W_0$.
%\item[(b)] The second crepant resolution $Y^+$ is the fine moduli space for the $\widetilde{\Gamma}$-representations $W = (W_i)_{i \in \widetilde{\Gamma}_0}$ of dimension vector $(1,1, \dots, 1)$ that is generated by the last component $W_{N-1}$.
%\end{enumerate}
%\end{thm}
We show that there is a natural correspondence between stable representations of Nakajima's quiver of type $\mathrm{A}_1$ and representations of $\widetilde{\Gamma}$.
At the end of Section  \ref{section, moduli}, we also characterize simple representations of the quiver,
namely we show that a simple representation corresponds to a point of a crepant resolution that lies over a non-singular point of $\overline{B(1)}$ (see Theorem \ref{thm simple}).

We note that these relations between a crepant resolution $Y$ (or $Y^+$) and an NCCR $\Lambda_k$ can be considered as a generalization of \textit{McKay correspondence}.
Classical McKay correspondence states that, for a finite subgroup $G \subset \SL_2$, there are many relations between the geometry of a quotient variety $\mathbb{C}^2/G$ and representations of the group $G$.
In the modern context, McKay correspondence is understood as relationships (e.g. a derived equivalence) between the crepant resolution $\widetilde{\mathbb{C}^2/G}$ of $\mathbb{C}^2/G$  and a quotient stack $[\mathbb{C}^2/G]$.
We often say that the crepant resolution $\widetilde{\mathbb{C}^2/G}$ is a ``geometric resolution" of $\mathbb{C}^2/G$.
On the other hand, since a coherent sheaf on a quotient stack $[\mathbb{C}^2/G]$ is canonically identified with a module over the skew group algebra $\mathbb{C}[x,y] \sharp G$, we say that a smooth stack $[\mathbb{C}^2/G]$ is an ``algebraic resolution" of $\mathbb{C}^2/G$.
Thus, we can interpret McKay correspondence as a correspondence between geometric and algebraic resolutions.
In our case, a geometric resolution of $\overline{B(1)}$ is $Y$ (or $Y^+$) and an algebraic resolution is the NCCR $\Lambda_k$.

\subsection{Mukai flops, P-twists and mutations}

It is well-known that the diagram of two crepant resolutions
\[ \begin{tikzcd}
Y \arrow[rd, "\phi"] & & Y^+ \arrow[ld, "\phi^+"'] \\
 & \overline{B(1)} &
\end{tikzcd} \]
is a local model of a class of flops that are called \textit{Mukai flop}.
%Since two varieties $Y$ and $Y^+$ are K-equivalent, they are expected to have the equivalent derived categories.
%A positive result for this expectation was given by Kawamata and Namikawa.
%According to the result of Kawamata and Namikawa, we have derived equivalences
%\[ \KN_k : \D(Y) \to \D(Y^+), ~~ \KN'_k : \D(Y^+) \to \D(Y) \]
%for each $k \in \mathbb{Z}$ as Fourier-Mukai functors (see Section \ref{subsect KN} for the concrete construction).
%The Fourier-Mukai kernel of $\KN_k$ and $\KN'_k$ are given as follow.
Let $\widetilde{Y}$ be a blowing-up of $Y$ along the zero-section $j(\PP(V)) \subset Y$.
Then, the exceptional divisor $E \subset \widetilde{Y}$ is naturally identified with the universal hyperplane in $\PP(V) \times \PP(V^*)$.
Let $\widehat{Y} := \widetilde{Y} \cup_E \PP(V) \times \PP(V^*)$ and $\mathcal{L}_k$ a line bundle on $\widehat{Y}$ such that $\mathcal{L}_k|_{\widetilde{Y}} = \stsh_{\widetilde{Y}}(kE)$ and $\mathcal{L}_k|_{\PP(V) \times \PP(V^*)} = \stsh(-k,-k)$.
By using a correspondence $Y \xleftarrow{\hat{q}} \widehat{Y} \xrightarrow{p} Y^+$, we define functors
\begin{align*}
\KN_k &:= R\hat{p}_*(L\hat{q}^*(-) \otimes \mathcal{L}_k) : \D(Y) \to \D(Y^+) \\
\text{and} ~~ \KN'_k &:= R\hat{q}_*(L\hat{p}^*(-) \otimes \mathcal{L}_k) : \D(Y^+) \to \D(Y).
\end{align*}
According to the result of Kawamata and Namikawa \cite{Ka02, Na03}, the functors $\KN_k$ and $\KN'_k$ give equivalences between $\D(Y)$ and $\D(Y^+)$.
On the other hand, by using tilting bundles $\Tilt_k$ and $\Tilt^+_{N-k-1}$ above, we get equivalences
\begin{align*}
\Psi_k : \D(Y) \xrightarrow{\sim}  \D(\modu(\Lambda_k)) ~~ \text{and} ~~ \Psi^+_{N-k-1} : \D(Y^+) \xrightarrow{\sim}  \D(\modu(\Lambda_k)).
\end{align*}
By composing $\Psi_k$ and the inverse of $\Psi^+_{N-k-1}$, we have an equivalence $\D(Y) \to \D(Y^+)$.
Although this functor seems to be different from the functor $\KN_k$ of Kawamata and Namikawa at a glance, we prove the following.
\begin{thm}[= Thm. \ref{thm kn=nkn}] \label{4th main thm}
Our functor $(\Psi^+_{N-k-1})^{-1} \circ \Psi_k$ (resp. $(\Psi_{N-k-1})^{-1} \circ \Psi^+_k$) coincides with the Kawamata-Namikawa's functor $\KN_k$ (resp. $\KN'_k$).
\end{thm}
We note that our proof of Theorem \ref{4th main thm} gives an alternative proof for the result of Kawamata and Namikawa that states the functors $\KN_k$ and $\KN'_k$ give equivalences of categories.

The $R$-algebras $\Lambda_{k}$ and $\Lambda_{k-1}$ are related by the operation that we call \textit{multi-mutation}.
We introduce a multi-mutation functor
\[ \nu^-_k : \D(\modu(\Lambda_k)) \to \D(\modu(\Lambda_{k-1})) \]
(see Definition \ref{defi mutation}) as an analog of Iyama-Wemyss's mutation functor \cite{IW14} (we call it IW mutation, for short) that Wemyss applied to his framework of ``Homological MMP'' for 3-folds (see \cite{We14}).
We show that a multi-mutation functor $\nu^-_k$ gives an equivalence of categories.
Moreover, we prove that our multi-mutation functor is obtained by composing IW mutation functors $N-1$ times (Theorem \ref{thank Wemyss}\footnote{This statement is suggested by Michael Wemyss in our private communication.}).
Dually, we introduce a multi-mutation functor $\nu^+_k : \D(\modu(\Lambda_k)) \to \D(\modu(\Lambda_{k+1}))$ and show that a multi-mutation $\nu^+_k$ is also a composition of $N-1$ IW mutation functors.
Whereas, it is well-known that the derived category $\D(Y)$ of a crepant resolution $Y$ has a non-trivial auto-equivalence called \textit{P-twist} (see Definition \ref{def P-twist}).
We show that a composition of multi-mutations corresponds to a P-twist on $\D(Y)$ in the following sense:
\begin{thm}[= Thm. \ref{thm mutation-Ptwist}]
Let 
\begin{align*}
\nu_{N+k}^- &: \D(\modu(\Lambda_{N+k})) \to \D(\modu(\Lambda_{N+k-1})) ~~ \text{and} \\
\nu_{N+k-1}^+ &: \D(\modu(\Lambda_{N+k-1})) \to \D(\modu(\Lambda_{N+k}))
\end{align*}
be multi-mutation functors.
Then we have the following diagram of equivalence functors commutes
\[ \begin{tikzcd}
\D(Y) \arrow{d}{P_k} \arrow{r}{\Psi_{N+k}} & \D(\modu(\Lambda_{N+k})) \arrow{d}{\nu^-_{N+k}}  \\
\D(Y) \arrow[r, "\Psi_{N+k-1}"] \arrow[d, equal] & \D(\modu(\Lambda_{N+k-1})) \arrow{d}{\nu^+_{N+k-1}} \\
\D(Y) \arrow[r, "\Psi_{N+k}"] & \D(\modu(\Lambda_{N+k})),
\end{tikzcd} \]
where $P_k : \D(Y) \to \D(Y)$ is the P-twist defined by a $\mathbb{P}^{N-1}$-object $j_*\stsh_{\mathbb{P}(V)}(k)$.
\end{thm}
This theorem means, under the identification $\Psi_{N+k} : \D(Y) \xrightarrow{\sim} \D(\modu(\Lambda_{N+k}))$, a composition of two multi-mutation functors
\[ \nu_{N+k-1}^+ \circ \nu_{N+k}^- \in \Auteq(\D(\modu(\Lambda_{N+k}))) \]
corresponds to a P-twist $P_k \in \Auteq(\D(Y))$.
Donovan and Wemyss proved that, in the case of three dimensional flops, a composition of two IW mutation functors corresponds to a spherical-like twist \cite{DW16}.
Our theorem says, in the case of Mukai flops, a composition of \textbf{many} IW mutations corresponds to a P-twist.
%(and so we call our result ``$\text{mutation}^n$-$\text{mutation}^n$ = twist" result).

As a corollary of the theorem above, we can prove the following functor isomorphism that was first proved by Cautis \cite{C12} and later by Addington-Donovan-Meachan \cite{ADM15}.
This result gives an example of ``flop-flop=twist" results that are widely observed \cite{To07, DW16, DW15}. 
\begin{cor}[=  \ref{cor CADM}, cf. \cite{ADM15, C12}]
We have a functor isomorphism
\[ \KN_{N+k} \circ \KN'_{-k} \simeq P_k \]
for all $k \in \mathbb{Z}$.
\end{cor}

%Cautis obtained this functor isomorphism by using an elaborate framework ``categorical $\mathfrak{sl}_2$-action" that is established by Cautis, Kamnitzer, and Licata \cite{CKL10, CKL13}.
%Addington, Donovan, and Meachan showed the same functor isomorphism in two different ways.
%In both ways, they embed Mukai flops into standard flops.
%The first way is the one that uses semi-orthogonal decompositions.
%This way is very close to the method of Bondal and Orlov's result that showed a standard flop induces an equivalence of derived categories.
%In the second way, they used the variation of GIT quotients and ``window shifts".

%Our approach that uses non-commutative crepant resolutions and mutations gives the 4th alternative proof for the above result.

\subsection{Plan of the article}
In Section \ref{section, prelim}, we provide some basic definitions and recall some fundamental results that we need in later sections.
%We also give some explanation for minimal nilpotent orbit closures of type A and its crepant resolutions.
In Section \ref{section, NCCR}, we construct an NCCR of a minimal nilpotent orbit closure of type A, and interpret it as the path algebra of a quiver.
%In the end of Section \ref{section, NCCR},
%we give another NCCR for $\overline{B(1)}$ and show that it is derived equivalent to the crepant resolutions and
%to the NCCR that we give in the first part of Section \ref{section, NCCR}.
In Section \ref{section, moduli}, we reconstruct the crepant resolutions from the quiver that gives the NCCR as moduli spaces of representations of the quiver. 
Furthermore, we study simple representations of the quiver.
% and show that the non-simple locus in the moduli coincides with the exceptional locus of the birational morphism that gives the crepant resolution.
In Section \ref{section, KN}, we study derived equivalences of the Mukai flop and P-twists on a crepant resolution via an NCCR.
%In the first half of this section, we compare the equivalence obtained by using the NCCR with Kawamata-Namikawa's equivalence.
%In the second half of the section, we introduce a derived equivalence between NCCRs that we call mutation,
%and show that it corresponds to a P-twist in the commutative side.

%\vspace{0.1in}
%\noindent
\subsection{Notations. } In this paper, we always work over the complex number field $\mathbb{C}$. Moreover, we adopt the following notations.

\begin{enumerate}
\item[$\bullet$] $V = \mathbb{C}^N$ : $N$-dimensional vector space over $\mathbb{C}$ ($N \geq 2$).
\item[$\bullet$] $\mathbb{P}(V) := V \setminus \{0\} / \mathbb{C}^{\times}$ : projectivization of a vector space $V$.
\item[$\bullet$] $\lvert E \rvert$ : the total space of a vector bundle $E$.
\item[$\bullet$] $\modu(A)$ : the category of finitely generated right $A$-modules.
\item[$\bullet$] $\D(\mathcal{A})$ : the (bounded) derived category of an abelian category $\mathcal{A}$.
\item[$\bullet$] $\D(X) := \D(\mathrm{coh}(X))$ : the derived category of coherent sheaves on a variety $X$.
\item[$\bullet$] $\Phi_{\mathcal{P}}$, $\Phi_{\mathcal{P}}^{X \to Y}$ : A Fourier-Mukai functor from $\D(X)$ to $\D(Y)$ whose kernel is $\mathcal{P} \in \D(X \times Y)$.
\item[$\bullet$] $\Sym^k_R M$ : $k$-th symmetric product of a $R$-module $M$.
\end{enumerate}

\vspace{0.1in}
\noindent
\textbf{Acknowledgments.}
The author would like to express his gratitude to his supervisor Professor Yasunari Nagai for beneficial conversations and helpful advices.
He encouraged me to tackle the problems studied in this paper.
The author would like to thank Professor Michael Wemyss for reading the previous version of this paper and suggesting Theorem \ref{thank Wemyss}.
The author also thanks Hiromi Ishii for many advice on drawing TikZ pictures.

%%%%%%%%%%%%%%%%%%%%%%%%%%%%%%%%%%%%%%%%%%%%%%%%%%%%%%%%%%%%%%%%%%%%%%%
\section{Preliminaries} \label{section, prelim}

\subsection{Non-commutative crepant resolutions}

\begin{defi} \label{def NCCR} \rm
Let $R$ be a Cohen-Macaulay (commutative) algebra and $M$ a non-zero reflexive $R$-module. We set $\Lambda := \End_R(M)$.
We say that the $R$-algebra $\Lambda$ is a \textit{non-commutative crepant resolution (=NCCR)} of $R$ or $M$ gives an NCCR of $R$ if 
\[ \gldim \Lambda_{\mathfrak{p}} = \dim R_{\mathfrak{p}} \]
for all $\mathfrak{p} \in \Spec R$ and $\Lambda$ is a (maximal) Cohen-Macaulay $R$-module.
\end{defi}

If we assume that $R$ is Gorenstein, we can relax the definition of NCCR.

\begin{lem}[\cite{IW14}] \label{def lem NCCR}
Let us assume that $R$ is Gorenstein and $M$ is a non-zero reflexive $R$-module.
In this case, an $R$-algebra $\Lambda := \End_R(M)$ is an NCCR of $R$ if and only if $\gldim \Lambda < \infty$ and $\Lambda$ is a (maximal) Cohen-Macaulay $R$-module.
\end{lem}

The theory of NCCR has strong relationship to the theory of tilting bundle.

\begin{defi} \rm
Let $X$ be a variety. A vector bundle $\Tilt$ (of finite rank) on $X$ is called a \textit{tilting bundle} if
\begin{enumerate}
\item[(1)] $\Ext_X^i(\Tilt, \Tilt) = 0$ for $i \neq 0$.
\item[(2)] $\Tilt$ classically generates the category $\mathrm{D}(\Qcoh(X))$, i.e. for $E \in \mathrm{D}(\Qcoh(X))$, $\RHom_X(\Tilt, E) = 0$ implies $E = 0$.
\end{enumerate}
\end{defi}

\begin{eg} \rm
In \cite{Bei79}, Beilinson showed that the following vector bundles on a projective space $\mathbb{P}^n$
\[ T = \bigoplus_{k=0}^n \stsh_{\mathbb{P}^n}(k), ~~~ T' = \bigoplus_{k=0}^n \Omega_{\mathbb{P}^n}^k(k+1) \]
are tilting bundles.
Note that these tilting bundles come from full strong exceptional collections of the derived category $\D(\PP^n)$ of $\PP^n$ that are called the \textit{Beilinson collections}.
\end{eg}

Once we fined a tilting bundle on a variety, we can construct an equivalence between the derived category of the variety and the derived category of a non-commutative algebra that is given as the endomorphism ring of the tilting bundle.
This is a generalization of classical Morita theory.

\begin{thm} \label{equiv tilting}
Let $\Tilt \in \D(X)$ be a tilting bundle on a smooth quasi-projective variety $X$. 
If we set $\Lambda := \End_X(\Tilt)$, we have an equivalence of categories
\[ \RHom_X(\Tilt, -) : \D(X) \xrightarrow{\sim} \D(\modu(\Lambda)), \]
and the quasi-inverse of this functor is given by
\[ - \otimes_{\Lambda} \Tilt : \D(\modu(\Lambda)) \xrightarrow{\sim} \D(X). \]
\end{thm}

For the proof of Theorem \ref{equiv tilting}, see \cite[Theorem 7.6]{HV07} or \cite[Lemma 3.3]{TU10}.

The following conjecture is due to Bondal, Orlov, and Van den Bergh.

\begin{conj}[\cite{VdB04b}, Conjecture 4.6] \label{BOV}
Let $R$ be a Gorenstein $\mathbb{C}$-algebra.
Then, all crepant resolutions of $R$ and all NCCRs of $R$ are derived equivalent.
\end{conj}

Van den Bergh showed that Conjecture \ref{BOV} holds if $R$ is of dimension $3$ and has only terminal singularities \cite{VdB04a, VdB04b}.
The existence of an NCCR and a derived equivalence between crepant resolutions and NCCRs are studied in many literatures \cite{Boc12, BLV10, Da10, HN17, Kal08, SV15, SV15b, SV17, TU10}.

In the rest of this subsection, we recall the basic property of reflexive modules.

\begin{lem}[\cite{BH93}, Proposition 1.4.1]
Let $R$ be a noetherian ring and $M$ a finitely generated $R$-module.
Then the following are equivalent.
\begin{enumerate}
\item[(1)] The module $M$ is reflexive,
\item[(2)] For each $\mathfrak{p} \in \Spec R$, one of the following happens
\begin{enumerate}
\item[(a)] $\depth(R_{\mathfrak{p}}) \leq 1$ and $M_{\mathfrak{p}}$ is a reflexive $R_{\mathfrak{p}}$-module, or
\item[(b)] $\depth(R_{\mathfrak{p}}) \geq 2$ and $\depth(M_{\mathfrak{p}}) \geq 2$.
\end{enumerate}
\end{enumerate}
\end{lem}

By using this lemma, we have the following.

\begin{prop} \label{CM ref 1}
Let $R$ be a normal Cohen-Macaulay domain and $M$ a (maximal) Cohen-Macaulay $R$-module.
Then, $M$ is reflexive.
\end{prop}

\begin{proof}
%We may assume that $R$ is local.
Let $\mathfrak{p}$ be a prime ideal of $R$.
If $\dim R_{\mathfrak{p}} \leq 1$, then the ring $R_{\mathfrak{p}}$ is regular and hence $M_{\mathfrak{p}}$ has finite projective dimension.
Therefore, by the Auslander-Buchsbaum formula (\cite[Theorem 1.3.3]{BH93})
\[ \mathrm{proj.dim}(M_{\mathfrak{p}}) + \depth M_{\mathfrak{p}} = \dim R_{\mathfrak{p}}, \]
$M_{\mathfrak{p}}$ is projective and hence free.
If $\dim R_{\mathfrak{p}} \geq 2$, we have $\depth(M_{\mathfrak{p}}) \geq 2$ by the assumption.
\end{proof}

\begin{prop} \label{CM ref 2}
Let $R$ be a normal Cohen-Macaulay domain and $M, N$ (maximal) Cohen-Macaulay $R$-modules.
Then, the $R$-module $\Hom_R(N, M)$ is reflexive.
\end{prop}

\begin{proof}
If $\dim R_{\mathfrak{p}} \leq 1$, then $M_{\mathfrak{p}}$ and $N_{\mathfrak{p}}$ are free and hence $\Hom_R(N,M)_{\mathfrak{p}}$ is also free.
Next we  assume that $R$ is local and $\dim R \geq 2$.
Then, it is enough to show that the depth of $\Hom_R(N, M)$ is greater than or equal to $2$.
Let us consider the resolution of $N$
\[ R^{\oplus N_1} \xrightarrow{\varphi} R^{\oplus N_0} \to N \to 0. \]
By applying the functor $\Hom_R(-,M)$, we have an exact sequence
\[ 0 \to \Hom_R(N,M) \to M^{\oplus N_0} \xrightarrow{\varphi^*} M^{\oplus N_1} \to \mathrm{coker}(\varphi^*) \to 0. \]
Then, by using Depth Lemma twice, we have the result.
\end{proof}

\subsection{Nilpotent orbit closures}
In this subsection, we recall some basic properties of nilpotent orbit closures.
The singularity of nilpotent orbit closures gives an important class of symplectic singularities (see \cite{Be00}).
First, we recall the notion of symplectic singularity.

\begin{defi}[\cite{Be00}] \rm
Let $X$ be an algebraic variety. We say that $X$  is a \textit{symplectic variety} if
\begin{enumerate}
\item[(i)] $X$ is normal.
\item[(ii)] The smooth part $X_{\mathrm{sm}}$ of $X$ admits a symplectic $2$-form $\omega$.
\item[(iii)] For every resolution $f : Y \to X$, the pull back of $\omega$ to $f^{-1}(X_{\mathrm{sm}})$ extends to a global holomorphic $2$-form on $Y$.
\end{enumerate}
Let $X$ be an algebraic variety. We say that a point $x \in X$ is a \textit{symplectic singularity} if there is an open neighborhood $U$ of $x$ such that $U$ is a symplectic variety.  
\end{defi}

Symplectic singularities belong to a good class of singularities that appears in minimal model theory.

\begin{prop}[\cite{Be00}]
A symplectic singularity is Gorenstein canonical.
\end{prop}

For symplectic singularities, we can consider the following reasonable class of resolutions.

\begin{defi} \rm
Let $X$ be a symplectic variety.
A resolution $\phi : Y \to X$ of $X$ is called \textit{symplectic} if the extended $2$-form $\omega$ on $Y$ is non-degenerate.
In other words, the $2$-form $\omega$ defines a symplectic structure on $Y$.
\end{defi}

\begin{prop}
Let $X$ be a symplectic variety and $\phi : Y \to X$ a resolution. Then, the following statements are equivalent
\begin{enumerate}
\item[(1)] $\phi$ is a crepant resolution,
\item[(2)] $\phi$ is a symplectic resolution,
\item[(3)] the canonical divisor $K_Y$ of $Y$ is trivial.
\end{enumerate}
\end{prop}

Next, we recall the definition of nilpotent orbit closures and some basic properties of them.
Let $\mathfrak{g}$ be a complex Lie algebra. For $u \in \mathfrak{g}$, we define a linear map $\ad_u : \mathfrak{g} \to \mathfrak{g}$ by $x \mapsto [u, x]$.
In the following, we assume that the Lie algebra $\mathfrak{g}$ is semi-simple, 
i.e. the bilinear form $\kappa(u,v) := \mathrm{trace}(\ad_u \circ \ad_v)$ is non-degenerate.
An element $v \in \mathfrak{g}$ is \textit{nilpotent} if the corresponding linear map $\ad_v$ is nilpotent.
Let $G$ be the adjoint algebraic group of $\mathfrak{g}$.
Then, $G$ acts on $\mathfrak{g}$ via the adjoint representation.
An orbit $\stsh = G \cdot v \subset \mathfrak{g}$ of $v$ under this action is called a \textit{nilpotent orbit} if the element $v$ is nilpotent.

\begin{prop}[\cite{Pa91}]
The normalization $\widetilde{\stsh}$ of a nilpotent orbit closure $\overline{\stsh}$ in a complex semi-simple Lie algebra $\mathfrak{g}$
has only symplectic singularities. 
Hence, the singularity of $\widetilde{\stsh}$ is Gorenstein canonical.
\end{prop}

Let $V = \mathbb{C}^N$ be a $N$-dimensional vector space and
\[ B(r) := \{ X \in \End_{\mathbb{C}}(V) \mid X^2 = 0, \rank(X) = r. \} \subset \mathfrak{sl}(V) \simeq \mathfrak{sl}_N. \]
This is a nilpotent orbit of type A.
We note that we have
\[ \overline{B(r)} = \{ X \in \End_{\mathbb{C}}(V) \mid X^2 = 0, \rank(X) \leq r \} = \bigcup_{k = 1}^r B(r). \]

If we consider nilpotent orbit closures of type A, we need not to take the normalization.

\begin{prop}[\cite{KF79}] \label{nilp Gore cano}
Let $ r \geq 1$.
Then, the variety $\overline{B(r)}$ is normal, and hence has only symplectic singularities.
In particular, the variety $\overline{B(r)}$ is Gorenstein, and has only canonical (equivalently, rational) singularities.
\end{prop}

Moreover, we can show that the variety $\overline{B(r)}$ has symplectic (equivalently, crepant) resolutions.

In the later sections, we study the case $r=1$. 

\subsection{The variety $\overline{B(1)}$ and its crepant resolutions $Y$ and $Y^+$}

Let $V = \mathbb{C}^N$ be an $N$-dimensional vector space and $\End_{\mathbb{C}}(V)$ an endomorphism ring of $V$.
Then, the $\SL_N := \SL(N, \mathbb{C})$ acts on $\End_{\mathbb{C}}(V)$ via the adjoint representation
\[ \Adj : \SL_N \to \GL(\End_{\mathbb{C}}(V)), ~~ A \mapsto (X \mapsto AXA^{-1}). \]
Let $X_0$ be an matrix in $B(1)$ such that
\[ X_0 := \begin{pmatrix}
0 & 0 & \cdots & 0 \\
\vdots & \vdots & \ddots & \vdots \\
0 & 0 & \cdots & 0 \\
1 & 0 & \cdots & 0
\end{pmatrix} \in \End_{\mathbb{C}}(V). \]
Then, we have
\[ \SL_N \cdot X_0 = B(1). \]

In the following, we consider homogeneous vector bundles on the orbit $B(1)$.
They correspond to linear representations of the stabilizer subgroup $\Stab_{\SL_N}(X_0)$ of $\SL_N$.

\begin{lem} \label{lem 2-stab}
The stabilizer subgroup $\Stab_{\SL_N}(X_0)$ is given by
\[ \Stab_{\SL_N}(X_0) = \left\{  \left( \begin{array}{c|ccc|c}
c & 0 & \cdots & 0 & 0 \\ \hline
 &  &  & & 0 \\
\ast &  &  A & & \vdots \\
 & & & & 0 \\ \hline
\ast &  & \ast &  & c
\end{array} \right) \mid 
\begin{array}{c}
A \in \GL_{N-2}, \\
c \in \mathbb{C} \setminus \{0 \}, \\
c^2 \cdot \det(A) = 1 
\end{array}
\right\}. \]
\end{lem}

\begin{proof}
Let $A = (a_{ij}) \in \SL_N$.
Then, we have
\begin{align*}
AX_0 = \left( \begin{array}{cccc}
a_{1N} & 0 & \cdots & 0 \\ 
a_{2N} & 0 & \cdots & 0  \\
 \vdots & \vdots & \ddots  & \vdots \\
a_{NN} & 0 & \cdots & 0 
\end{array} \right),  ~~
X_0A = \left( \begin{array}{cccc}
0 & 0 & \cdots & 0 \\ 
 \vdots & \vdots & \ddots  & \vdots \\
0 & 0 & \cdots & 0 \\
a_{11} & a_{12} & \cdots & a_{1N} 
\end{array} \right).
\end{align*}
Thus, if $AX_0 = X_0A$, we have $a_{11} = a_{NN}$, $a_{12} = \cdots = a_{1N} = 0$, and $a_{2N} = \cdots = a_{N-1,N} = 0$.
\end{proof}

\begin{defi} \rm
\begin{enumerate}
\item[(1)] For $a \in \mathbb{Z}$, we define a character $m_a : \Stab_{\SL_N}(X_0) \to \mathbb{C}^{\times}$ as 
\[ \Stab_{\SL_N}(X_0) \ni \left( \begin{array}{c|ccc|c}
c & 0 & \cdots & 0 & 0 \\ \hline
 &  &  & & 0 \\
\ast &  &  A & & \vdots \\
 & & & & 0 \\ \hline
\ast &  & \ast &  & c
\end{array} \right) \mapsto c^{-a} \in \mathbb{C}^{\times}. \]
\item[(2)] Let $\mathcal{M}_a$ be a line bundle on $B(1)$ that corresponds to the character $m_a$.
\item[(3)] We set $M_a := H^0(B(1), \mathcal{M}_a)$. Then $M_a$ is a reflexive $R$-module.
\end{enumerate}
\end{defi}

Next, let us consider a resolution $Y$ of $\overline{B(1)}$.
The resolution $Y$ is given by
\[ Y := \{ (X, L) \in \End_{\mathbb{C}}(V) \times \mathbb{P}(V) \mid X(V) \subset L, X^2 = 0 \} \]
and a left $\SL_N$-action on $Y$ is given by
\[ A \cdot (X, L) := (AXA^{-1}, AL) \]
for $A \in \SL_N$ and $(X,L) \in Y$.
Via the second projection $\pi : Y \to \mathbb{P}(V)$, one can see that $Y$ is isomorphic to the total space of the cotangent bundle $\Omega_{\mathbb{P}(V)}$ on $\mathbb{P}(V)$.
Note that the embedding $Y \subset \End_{\mathbb{C}}(V) \times \mathbb{P}(V)$ is determined by a composition of injective bundle maps
\[ \Omega_{\mathbb{P}(V)} \subset V^* \otimes_{\mathbb{C}} \stsh_{\mathbb{P}(V)}(-1) \subset V^* \otimes_{\mathbb{C}} V \otimes_{\mathbb{C}} \stsh_{\mathbb{P}(V)}. \]
Let $j : {\mathbb{P}(V)} \to Y$ be the zero-section, and then $j({\mathbb{P}(V)})$ is given by
\[ j({\mathbb{P}(V)}) = \{ (0, L) \in \End_{\mathbb{C}}(V) \times \PP(V) \}. \]
On the other hand, the image of the first projection $\phi : Y \to \End_{\mathbb{C}}(V)$ is just $\overline{B(1)}$,
and if we set $U := Y \setminus j({\mathbb{P}(V)})$,
then, $\phi$ contracts $j({\mathbb{P}(V)})$ to a point $0 \in \overline{B(1)}$, and $U$ is isomorphic to $B(1)$ via the morphism $\phi : Y \to \overline{B(1)}$.
Thus, the first projection $\phi$ gives a resolution of $\overline{B(1)}$.
Since the affine variety $\overline{B(1)}$ is a symplectic variety,
the canonical divisor of $\overline{B(1)}$ is trivial.
On the other hand, since $Y$ is isomorphic to the total space of the cotangent bundle on a projective space, the canonical divisor of $Y$ is also trivial.
Thus, the resolution of singularities $\phi : Y \to \overline{B(1)}$ is a crepant resolution, and in this case, is symplectic resolution of $\overline{B(1)}$.

Let us set $\stsh_Y(a) := \pi^*\stsh_{\mathbb{P}(V)}(a)$.

\begin{lem} \label{thm1 lem}
Under the identification $U \simeq B(1)$, the homogeneous vector bundle $\mathcal{M}_a$ is isomorphic to $\stsh_Y(a)|_U$.
\end{lem}

\begin{proof}
We first note that $\stsh_Y(a)|_U$ is a homogeneous line bundle on $U$.
Let $L_0 := X_0(V)$, then $L_0$ is a line in $V$.
Let $y_0 := (X_0, L_0) \in U$ be a point.
The fiber of the line bundle $\stsh_Y(a)|_U$ at $y_0 \in U$ is canonically isomorphic to $L_0^{\otimes -a}$.
Note that the action of $\Stab_{\SL_N}(X_0)$ on $L_0$ is given by
\[ \left( \begin{array}{c|ccc|c}
c & 0 & \cdots & 0 & 0 \\ \hline
 &  &  & & 0 \\
\ast &  &  A & & \vdots \\
 & & & & 0 \\ \hline
\ast &  & \ast &  & c
\end{array} \right) \cdot 
\left( \begin{array}{c}
0 \\
0 \\
\vdots \\
0 \\
a_N
\end{array} \right)
= \left( \begin{array}{c}
0 \\
0 \\
\vdots \\
0 \\
ca_N
\end{array} \right)
\]
Therefore, the character $\Stab_{\SL_N}(X_0) \to \GL(L_0^{\otimes -a})$ that determines $\stsh_Y(a)|_U$ coincides with the one that defines $\mathcal{M}_a$.
\end{proof}

Next, we study the other crepant resolution $Y^+$ of $\overline{B(1)}$.
Let $\PP(V^*)$ be a dual projective space, that is
\[ \mathbb{P}(V^*) = \{ H \subset V \mid \text{$H$ is a hyperplane in $V$} \}. \]
The variety $Y^+$ is defined by
\[ Y^+ := \{ (X, H) \in  \End_{\mathbb{C}}(V) \times \PP(V^*) \mid X(V) \subset H, X(H) = 0 \}. \]
An $\SL_N$-action on $Y^+$ is given by $A \cdot (X, H) = (AXA^{-1}, AH)$.
Let $\phi^+ : Y \to \overline{B(1)}$ be the first projection and $\pi' : Y^+ \to \PP(V^*)$ the second projection.
As in the case of $Y$, 
$Y^+$ is isomorphic to the total space of the cotangent bundle $\Omega_{\PP(V^*)}$ on $\PP(V^*)$ via the second projection $\pi' : Y^+ \to \PP(V^*)$,
and the first projection $\phi^+ : Y^+ \to \overline{B(1)}$ gives a crepant resolution of $\overline{B(1)}$.
The morphism $\phi^+ : Y \to \overline{B(1)}$ contracts the zero section $j' : \mathbb{P}(V^*) \hookrightarrow Y^+$.
Let $U^+ := Y^+ \setminus j'(\PP(V^*))$ and $\stsh_{Y^+}(a) := (\pi')^*\stsh_{\PP(V^*)}(a)$.

As in the above, we can show the following.

\begin{lem} \label{thm1 lem2}
Under the identification $U^+ \simeq B(1)$, the homogeneous vector bundle $\mathcal{M}_a$ is isomorphic to $\stsh_{Y^+}(-a)|_{U^+}$.
\end{lem}

\subsection{Iyama-Wemyss's mutation} \label{subsect IW}

In the present subsection, we recall some basic definitions and properties about Iyama-Wemyss's mutation.

\begin{defi} \rm
Let $R$ be a $d$-singular Calabi-Yau ring\footnote{We do not give the definition here but note that this is equivalent to say that $R$ is Gorenstein and $\dim R_{\mathfrak{m}} = d$ for all maximal ideal $\mathfrak{m} \subset R$ \cite{IR08}.} ($d$-sCY, for short).
A reflexive $R$-module $M$ is say to be a \textit{modifying module} if $\End_{R}(M)$ is a (maximal) Cohen-Macaulay $R$-module.
\end{defi}

\begin{defi} \rm
Let $A$ be a ring, $M, N$ $A$-modules, and $N_0 \in \add N$.
A morphism $f : N_0 \to M$ is called a \textit{right $(\add N)$-approximation} if the map
\[ \Hom_A(N, N_0) \xrightarrow{f \circ} \Hom_A(N,M) \]
is surjective.
\end{defi}

Let $R$ be a normal $d$-sCY ring and $M$ a modifying $R$-module.
For $0 \neq N \in \add M$, we consider 
\begin{enumerate}
\item[(1)] a right $(\add N)$-approximation of $M$, $a : N_0 \to M$.
\item[(2)] a right $(\add N^*)$-approximation of $M^*$, $b : N_1^* \to M^*$.
\end{enumerate}
Let $K_0 := \Ker(a)$ and $K_1 := \Ker(b)$.

\begin{defi} \rm
With notations as above, we define the \textit{right mutation} of $M$ at $N$ to be $\mu_N^R(M) := N \oplus K_0$
and the \textit{left mutation} of $M$ at $N$ to be $\mu_N^L(M) := N \oplus K_1^*$.
\end{defi}

In \cite{IW14}, Iyama and Wemyss proved the following theorem.
 
\begin{thm}[\cite{IW14}] \label{IW mutation equiv}
Let $R$ be a normal $d$-sCY ring
and $M$ a modifying module.
Assume that $0 \neq N \in \add M$.
Then
\begin{enumerate}
\item[(1)] $R$-algebras $\End_R(M)$, $\End_R(\mu_N^R(M))$, and $\End_R(\mu^L_N(M))$ are derived equivalent.
\item[(2)] If $M$ gives an NCCR of $R$, so do its mutations $\mu_N^R(M)$ and $\mu^L_N(M)$.
\end{enumerate}
\end{thm}

The equivalence between $\End_R(M)$ and $\End_R(\mu^L_N(M))$ is given as follows.
Let $Q := \Hom_R(M, N)$ and 
\[ V := \mathrm{Image}\left(\Hom_R(M, N_1) \to \Hom_R(M, K_1^*)\right). \]
Then, one can show that $V \oplus Q$ is a tilting $\Lambda := \End_R(M)$-module and there is an isomorphism of $R$-algebras
\[ \End_R(\mu_N^L(M)) \simeq \End_{\Lambda}(V \oplus Q). \]
Thus, we have an equivalence
\[ \IW_N := \RHom(V \oplus Q, -) : \D(\modu(\End_R(M))) \to \D(\modu(\End_R(\mu_N^L(M)))). \]
In this paper, we only use left IW mutations and hence we call them simply the \textit{IW mutation}.
We also call the functor $\IW_N$ the \textit{IW mutation functor}.

In the later section, we introduce a concept of \textit{multi-mutations} and prove that a multi-mutation can be written as a composition of IW mutation functors.

\subsection{P-twists}

In this subsection, we recall the definition of P-twists and their basic properties.

\begin{defi} \label{def P-twist} \rm
An object $E$ in the derived category $\D(X)$ of a variety $X$ of dimension $2n$ is called a \textit{$\mathbb{P}$-object} if we have $E \otimes \omega_X \simeq E$ and
\[ \Hom(E,E[i]) \simeq H^i(\mathbb{P}^n; \mathbb{C}) \]
for all $i \in \mathbb{Z}$.
For a $\mathbb{P}$-object $E$, the \textit{P-twist} $P_E : \D(X) \to \D(X)$ by $E$ is defined as follow
\[ P_E(F) := \Cone\left(\Cone(E \otimes \RHom(E,F)[-2] \to E \otimes \RHom(E,F)) \xrightarrow{\mathrm{ev}} F \right). \]
\end{defi}

See Lemma \ref{P-obj lem} for a basic example of $\PP$-object.

\begin{prop}[\cite{HT06}]
A P-twist gives an auto-equivalence of $\D(X)$.
\end{prop}

The notion of P-twist was first introduced by Huybrechts and Thomas in their paper \cite{HT06} as an analogue of the notion of \textit{spherical twist}.
Spherical twists give an important class of auto-equivalences on the derived category of a Calabi-Yau variety.
In contrast, P-twists give a significant class of auto-equivalences on the derived category of a (holomorphic) symplectic variety.
In Section \ref{subsect mutation}, we study P-twists on the symplectic variety $Y = \lvert \Omega_{\mathbb{P}(V)} \rvert$ that is explained in the above section, from the view point of NCCR.

%%%%%%%%%%%%%%%%%%%%%%%%%%%%%%%%%%%%%%%%%%%%%%%%%%%%%%%%%%%%%%%%%%%%%%%
\section{Non-commutative crepant resolutions of $\overline{B(1)}$} \label{section, NCCR}

\subsection{The existence of NCCRs of $\overline{B(1)}$ and relations between CRs}

In this section, we study non-commutative crepant resolutions of a minimal nilpotent closure $\overline{B(1)} \subset \End(V)$ where $V = \mathbb{C}^N$.
We always assume $N \geq 2$.
Let $R$ be the affine coordinate ring of $\overline{B(1)}$.
By Proposition \ref{nilp Gore cano}, the $\mathbb{C}$-algebra $R$ is Gorenstein and normal.
Note that $\overline{B(1)} = B(1) \cup \{0\}$ as set and hence we have
\[ \codim_{\overline{B(1)}}(\overline{B(1)} \setminus B(1)) = N \geq 2. \]
Thus, we have a $\mathbb{C}$-algebra isomorphism
\[ R \simeq H^0(B(1), \stsh_{B(1)}). \]

\begin{lem} \label{lem 3-A}
Let $\mathcal{F}$ be a coherent sheaf on $Y$ that satisfies
\begin{enumerate}
\item[(a)] $\Ext_Y^i(\mathcal{F}, \stsh_Y) = 0$ for $i > 0$, and
\item[(b)] $H^i(Y, \mathcal{F}) = 0$ for $i > 0$.
\end{enumerate}
Then, the push-forward $\phi_*\mathcal{F} =: M$ is a Cohen-Macaulay $R$-module.
\end{lem}

\begin{proof}
Since the resolution $\phi : Y \to \overline{B(1)} = \Spec R$ is crepant, we have $\phi^!\stsh_{\overline{B(1)}} \simeq \stsh_Y$.
Thus, we have
\begin{align*}
\Ext_Y^i(\mathcal{F}, \stsh_Y) &\simeq \Ext_Y^i(\mathcal{F}, \phi^!\stsh_{\overline{B(1)}}) \\
&\simeq \Ext^i_R(R\phi_*\mathcal{F}, R) \\
&\simeq \Ext_R^i(M, R)
\end{align*}
and hence we have
\[ \Ext_R^i(M, R) = 0 \]
for $i > 0$.
Let $\mathfrak{m} \subset R$ be a maximal ideal that corresponds to the origin $0 \in \overline{B(1)}$,
$(\hat{R}, \hat{\mathfrak{m}})$ the $\mathfrak{m}$-adic completion of $(R_{\mathfrak{m}}, \mathfrak{m})$,
and $\hat{M}$ the $\mathfrak{m}$-adic completion of $M_{\mathfrak{m}}$.
Since the local algebra $\hat{R}$ is Gorenstein,
the canonical module $\omega_{\hat{R}}$ is isomorphic to $\hat{R}$ as an $\hat{R}$-module.
Thus, by Grothendieck's local duality theorem (see \cite[Theorem 3.5.8]{BH93}), we have
\[ H_{\hat{\mathfrak{m}}}^i(\hat{M}) = \Hom_{\hat{R}}\left(\Ext^{2N-i-2}_{\hat{R}}(\hat{M}, \hat{R}), E(\hat{R}/\hat{\mathfrak{m}})\right) \]
where $E(\hat{R}/\hat{\mathfrak{m}})$ is the injective hull of the residue field $\hat{R}/\hat{\mathfrak{m}}$.
Therefore, we have
\[ H_{\hat{\mathfrak{m}}}^i(\hat{M}) = 0 \]
for $i < 2N-2 = \dim R$ and hence $M$ is a Cohen-Macaulay $R$-module (see \cite[Theorem 3.5.7]{BH93}).
\end{proof}

\begin{lem}  \label{lem 3-B}
Let $\mathcal{E}$ be a vector bundle on $\mathbb{P}(V)$ such that 
\[ H^i(\mathbb{P}(V), \mathcal{E}(k)) = 0 \]
for all $i > 0$ and $k \geq 0$.
Then, we have
\[ H^i(Y, \pi^*\mathcal{E}) = 0 \]
for all $i > 0$.
\end{lem}

\begin{proof}
Let $Z$ be a total space of a vector bundle $V^* \otimes_{\mathbb{C}} \stsh_{\mathbb{P}(V)}(-1)$.
Then, $Y$ is embedded in $Z$ via the Euler sequence
\[ 0 \to \Omega_{\mathbb{P}(V)} \to V^* \otimes_{\mathbb{C}} \stsh_{\mathbb{P}(V)}(-1) \to \stsh_{\mathbb{P}(V)} \to 0. \]
Since $(V \otimes_{\mathbb{C}} \stsh_{\mathbb{P}(V)}(-1))/\Omega_{\mathbb{P}(V)} \simeq \stsh_{\mathbb{P}(V)}$, the ideal sheaf $I_{Y/Z}$ is isomorphic to $\stsh_Z$.
Thus, we have an exact sequence on $Z$
\[ 0 \to \stsh_Z \to \stsh_Z \to \stsh_Y \to 0. \]
Let $\pi_Z : Z \to {\mathbb{P}(V)}$ be the projection.
Then, we have
\begin{align*}
H^i(Z, \pi_Z^* \mathcal{E}) &\simeq H^i({\mathbb{P}(V)}, \mathcal{E} \otimes R{\pi_Z}_* \stsh_Z) \\
&\simeq H^i({\mathbb{P}(V)}, \mathcal{E} \otimes {\pi_Z}_*\stsh_Z) ~~~(\text{since $\pi_Z$ is affine}) \\
&\simeq \bigoplus_{k \geq 0} \Sym^k V \otimes_{\mathbb{C}} H^i({\mathbb{P}(V)}, \mathcal{E}(k))
\end{align*}
and this is zero for $i > 0$ by the assumption.
Thus we have
\[ H^i(Z, \pi_Z^* \mathcal{E} \otimes \stsh_Y) = H^i(Y, \pi^*\mathcal{E}) = 0 \]
for $i > 0$.
\end{proof}

\begin{defi} \rm
For an integer $k \in \mathbb{Z}$, let $\Tilt_k := \bigoplus_{a = -N+k+1}^k \stsh_Y(a)$ be a vector bundle on $Y$ and $\Lambda_k := \End_Y(\Tilt_k)$ the endomorphism ring of $\Tilt_k$.
\end{defi}

Note that the $R$-algebra structure of $\Lambda_k$ does not depend on the choice of the integer $k$.
Nevertheless, we adopt this notation to emphasize that the algebra $\Lambda_k$ is given as the endomorphism ring of a bundle $\Tilt_k$.

\begin{thm} \label{NCCR1}
The following hold.
\begin{enumerate}
\item[(1)] For all $k \in \mathbb{Z}$, the vector bundle $\Tilt_k$ is a tilting bundle on $Y$.
\item[(2)] For all $-N+1 \leq a \leq N-1$, we have 
\[ \phi_*\stsh_Y(a) = M_a, \]
and $M_a$ is a (maximal) Cohen-Macaulay $R$-module.
\item[(3)] If $0 \leq k \leq N-1$, then we have an isomorphism
\[ \End_Y(\Tilt_k) \simeq \End_R\left(\bigoplus_{a=-N+k+1}^k M_a\right) \]
\item[(4)] The $R$-module 
\[ \bigoplus_{a=-N+k+1}^k M_a \]
gives an NCCR $ \Lambda_k$ of $R$
%\[ \Lambda_k :=  \End_R\left(\bigoplus_{a=-N+k+1}^k M_a\right) \]
for $0 \leq k \leq N-1$.
\item[(5)] There is an equivalence of categories
\[ \RHom_Y(\Tilt_k, -) : \D(Y) \xrightarrow{\sim} \D(\modu(\Lambda_k)). \]
\end{enumerate}
\end{thm}

We note that (1) and (5) of Theorem \ref{NCCR1} are also obtained by Toda and Uehara in \cite{TU10}.
They also study the perverse heart of $\D(Y)$ that corresponds to $\modu(\Lambda_0)$ via the derived equivalence.

\begin{proof}
Let $T = \bigoplus_{a=0}^{N-1} \stsh_{\mathbb{P}(V)}(a)$ is a tilting bundle on $\mathbb{P}(V)$.
Then, we have
\[ H^i(\mathbb{P}(V), T^* \otimes T \otimes \stsh_{\mathbb{P}(V)}(k)) = 0 \]
for all $i > 0$ and $k \geq 0$.
Thus, by Lemma \ref{lem 3-B}, we have
\[ H^i(Y, \Tilt_0^* \otimes \Tilt_0) = 0 \]
and hence $\Tilt_0$ is a tilting bundle on $Y$.
Since other bundles $\Tilt_k$ ($k \in \mathbb{Z}$) are obtained from $\Tilt_0$ by twisting $\stsh_Y(k)$, $\Tilt_k$ ($k \in \mathbb{Z}$) are also tilting bundles on $Y$.
This shows (1).

On the other hand, by Lemma \ref{lem 3-B}, we have
\[ H^i(Y, \stsh_Y(a)) = 0 ~~\text{for $i>0$} \]
if $a \geq -N+1$.
Therefore, if $-N+1 \leq a \leq N-1$, we have
\begin{align*}
H^i(Y, \stsh_Y(a)) &= 0, \\
\Ext^i_Y(\stsh_Y(a), \stsh_Y) &= 0
\end{align*}
for all $i > 0$.
Thus, by Lemma \ref{lem 3-A}, we have the $R$-module $\phi_*\stsh_Y(a) = H^0(Y, \stsh_Y(a))$ is Cohen-Macaulay if $-N+1 \leq a \leq N-1$.
In particular, if $-N+1 \leq a \leq N-1$, $\phi_*\stsh_Y(a)$ is a reflexive $R$-module by Proposition \ref{CM ref 1}.
By Lemma \ref{thm1 lem}, $\phi_*\stsh_Y(a)$ and $M_a$ are isomorphic outside the unique singular point $0 \in \overline{B(1)}$.
Thus, we have $\phi_*\stsh_Y(a) \simeq M_a$ for $-N+1 \leq a \leq N-1$ and hence $M_a$ is (maximal) Cohen-Macaulay as an $R$-module if $-N+1 \leq a \leq N-1$.
This shows (2).

Next, we prove (3).
By Lemma \ref{lem 3-A} and (1), we have $\End_Y(\Tilt_k) \simeq \phi_*(\Tilt_k^* \otimes \Tilt_k)$ is a (maximal) Cohen-Macaulay $R$-module and hence is reflexive by Proposition \ref{CM ref 1}.
On the other hand, the $R$-module
\[ \End_R\left(\bigoplus_{a=-N+k+1}^k M_a\right) \]
is also reflexive for $0 \leq k \leq N-1$ by Proposition \ref{CM ref 2}.
These two reflexive $R$-modules are isomorphic to each other outside the unique singular point $0 \in \overline{B(1)}$.
Thus, we have
\[ \End_Y(\Tilt_k) \simeq \End_R\left(\bigoplus_{a=-N+k+1}^k M_a\right). \]
Finally, (4) follows from (1), (2), and (3). (5) follows from (1).
\end{proof}

It is easy to see that the dual statements hold for $Y^+$.

\begin{thm} \label{main thm flop side}
Let $\Tilt_k^+ := \bigoplus_{a=-N+k+1}^k \stsh_{Y^+}(a)$.
Then, the following hold.
\begin{enumerate}
\item[(1)] For all $k \in \mathbb{Z}$, the vector bundle $\Tilt^+_k$ is a tilting bundle on $Y^+$.
\item[(2)] For all $-N+1 \leq a \leq N-1$, we have 
\[ \phi^+_*\stsh_{Y^+}(a) = M_{-a}. \]
\item[(3)] If $0 \leq k \leq N-1$, then we have an isomorphism
\[ \End_{Y^+}(\Tilt^+_k) \simeq \End_R\left(\bigoplus_{a=-N+k+1}^k M_{-a} \right). \]
\item[(4)] For all $k \in \mathbb{Z}$, there is a canonical isomorphism
\[ \End_{Y^+}(\Tilt^+_k) \simeq \Lambda_{N-k-1}. \]
\item[(5)] There is an equivalence of categories
\[ \RHom_{Y^+}(\Tilt^+_k, -) : \D(Y^+) \xrightarrow{\sim} \D(\modu(\Lambda_{N-k-1})). \]
\end{enumerate}
\end{thm}

\begin{proof}
We only show (4).
By Lemma \ref{thm1 lem} and Lemma \ref{thm1 lem2}, we have $\Lambda_k = \End_Y(\Tilt_k)$ and $\End_{Y^+}(\Tilt^+_{N-k-1})$ are isomorphic to each other on the smooth locus $B(1)$.
Since both algebras are Cohen-Macaulay as $R$-modules and hence are reflexive, we have an isomorphism
\[ \Lambda_k = \End_Y(\Tilt_k) \simeq \End_{Y^+}(\Tilt^+_{N-k-1}). \]
This is what we want.
\end{proof}

%%%%%%%%%%%%%%%%%%%%%%%%%%%%%%%%%%%%%%%%%%%%%
\subsection{NCCRs as the path algebra of a quiver}
The aim of this subsection is to describe the NCCR $\Lambda_k$ of $\overline{B(1)}$ as the path algebra of a quiver with relations.

As in the above subsection, let $Z$ be the total space of a vector bundle $V^* \otimes \stsh_{\mathbb{P}(V)}(-1)$.
Let $\pi_Z : Z \to {\mathbb{P}(V)}$ the projection,
and we set $\stsh_Z(a) := \pi_Z^*\stsh_{\mathbb{P}(V)}(a)$, $\Tilt_Z := \bigoplus_{a=-N+1}^0 \stsh_Z(a)$, and $\Lambda_Z := \End_Z(\Tilt_Z)$.
Then, the algebra $\Lambda_k$ is a quotient algebra of $\Lambda_Z$.
First, we describe the non-commutative algebra $\Lambda_Z$ as the path algebra of a quiver with certain relations.

Note that $Z$ is a crepant resolution of an affine variety $\Spec H^0(Z, \stsh_Z)$.
We set $\widetilde{R} := H^0(Z, \stsh_Z)$.
Then, the algebra $\widetilde{R}$ is described as follows.
\begin{align*}
\widetilde{R} &:= H^0(Z, \stsh_Z) \\
&\simeq H^0(\mathbb{P}(V), \bigoplus_{k \geq 0} \Sym^k V \otimes_{\mathbb{C}} \stsh_{\mathbb{P}(V)}(k)) \\
&\simeq \bigoplus_{k \geq 0} \Sym^k V \otimes_{\mathbb{C}} \Sym^k V^*
\end{align*}
Let $S$ be the affine coordinate ring of $\End_{\mathbb{C}}(V)$, i.e. 
\[ S := \bigoplus_{k \geq 0} \Sym^k(V \otimes_{\mathbb{C}} V^*). \]
Let $v_1, \cdots, v_N$ be the standard basis of $V = \mathbb{C}^N$ and $f_1, \dots, f_N \in V^*$ the dual basis.
If we set $x_{ij} := v_j \otimes f_i$, the algebra $S$ is isomorphic to the polynomial ring with $N^2$ variables
\[ S \simeq \mathbb{C}[ (x_{ij})_{i,j = 1, \dots, N}]. \]
The affine variety $\Spec \widetilde{R}$ is embedded in $\End_{\mathbb{C}}(V) = \Spec S$ via the canonical surjective homomorphism of algebras
\[ S := \bigoplus_{k \geq 0} \Sym^k (V \otimes_{\mathbb{C}} V^*) \twoheadrightarrow \bigoplus_{k \geq 0} \Sym^k V \otimes_{\mathbb{C}} \Sym^k V^*. \]

Next, we define quivers that we use later.

\begin{defi} \rm
Let $\Gamma$ be the \textit{Beilinson quiver}
\[ \begin{tikzcd}
   0 \arrow[r, shift left=1, phantom, "\vdots" description] \arrow[r, shift left=5, "f_1" description] \arrow[r, shift right=5, "f_N" description] &
   1 \arrow[r, shift left=1, phantom, "\vdots" description] \arrow[r, shift left=5, "f_1" description] \arrow[r, shift right=5, "f_N" description] &
   \cdots \arrow[r, shift left=1, phantom, "\vdots" description] \arrow[r, shift left=5, "f_1" description] \arrow[r, shift right=5, "f_N" description] &
   N-2 \arrow[r, shift left=1, phantom, "\vdots" description] \arrow[r, shift left=5, "f_1" description] \arrow[r, shift right=5, "f_N" description] &
   N-1
\end{tikzcd} \]

and $\widetilde{\Gamma}$ the \textit{double Beilinson quiver}

\[ \begin{tikzpicture}[
 description/.style={midway,              
                     fill=white, 
                     font=\scriptsize,
                     inner sep=2pt,
                     minimum height=.8em
 }]
 \matrix[matrix of math nodes, column sep=1cm]{
      |(n1)| 0 & |(n2)| 1 & |(dots)| \cdots & |(N1)| N - 2 & |(N)| N  -1 \\
 };
 \foreach \f/\t in {n1/n2, n2/dots, dots/N1, N1/N}{ 
  \begin{scope}
   \coordinate (ulc) at ($(\f.south east)-(3pt,3pt)$);
     \coordinate (urc) at ($(ulc -| \t.west)+(3pt,0)$);
   \path[->,every node/.style={description}] 
    (\f.south east) edge[bend right=30] node (f1) {$f_1$} (\f.south east -| \t.west)
    (ulc) edge[bend right=70,looseness=2] node (fN) {$f_N$} (urc)
    ; 
   \path[draw=black, dotted, thick] (f1.south) to (fN.north);
   \end{scope}
  };
   \foreach \f/\t in {n2/n1, dots/n2, N1/dots, N/N1}{ 
   \begin{scope}
   \coordinate (alc) at ($(\f.north west)+(3pt,3pt)$);
   \coordinate (arc) at ($(alc -| \t.east)+(-3pt,0)$);
   \path[->,every node/.style={description}]
    (\f.north west) edge[bend right=30] node (vN) {$v_N$} (\f.north west -| \t.east)
    (alc) edge[bend right=70,looseness=2] node (v1) {$v_1$} (arc)
    ;
   \path[draw=black, dotted, thick] (vN.north) to (v1.south);
  \end{scope}
 };
\end{tikzpicture} \]
Here, $v_i, f_j$ serve as the label for $N$ different arrows.
\end{defi}

Next, we show that the non-commutative algebra $\Lambda_Z$ has a description as the path algebra of the double Beilinson quiver with certain relations.

\begin{thm} \label{NCCR quiver Z}
The non-commutative algebra $\Lambda_Z$ is isomorphic to the path algebra $S\widetilde{\Gamma}$ of the double Beilinson quiver $\widetilde{\Gamma}$ over $S$
with relations
\begin{align*}
v_iv_j &= v_jv_i ~~ \text{for all $1 \leq i, j \leq N$} \\
f_if_j &= f_jf_i ~~ \text{for all $1 \leq i, j \leq N$} \\
v_jf_i &= f_iv_j = x_{ij}~~ \text{for all $1 \leq i, j \leq N$}.
\end{align*}
\end{thm}

\begin{proof}
First, for $a , b \in \mathbb{Z}_{\geq 0}$ we have
\begin{align*}
\Hom_Z(\stsh_Z(a), \stsh_Z(b)) &\simeq \Hom_{\mathbb{P}(V)}(\stsh_{\mathbb{P}(V)}(a), (\pi_Z)_*\stsh_Z \otimes \stsh_{\mathbb{P}(V)}(b)) \\
&\simeq \Hom_{\mathbb{P}(V)}\left(\stsh_{\mathbb{P}(V)}(a), \left(\bigoplus_{k \geq 0} \Sym^k V \otimes_{\mathbb{C}} \stsh_{\mathbb{P}(V)}(1)\right) \otimes \stsh_{\mathbb{P}(V)}(b)\right). \\
&\simeq \Hom_{\mathbb{P}(V)}\left(\stsh_{\mathbb{P}(V)}(a), \bigoplus_{k \geq 0} \Sym^k V \otimes_{\mathbb{C}} \stsh_{\mathbb{P}(V)}(b + k) \right).
\end{align*}
Moreover, if $b \geq a$, we have
\begin{align*}
\Hom_Z(\stsh_Z(a), \stsh_Z(b)) \simeq \bigoplus_{k \geq 0} \Sym^k V \otimes_{\mathbb{C}} \Sym^{k+b - a} V^*,
\end{align*}
and if $b \leq a$, we have
\begin{align*}
\Hom_Z(\stsh_Z(a), \stsh_Z(b)) \simeq \bigoplus_{k \geq 0} \Sym^{k+a-b} V \otimes_{\mathbb{C}} \Sym^{k} V^*.
\end{align*}
We define the action $v : \stsh_Z(a) \to \stsh_Z(a-1)$ of $v \in V$ on $\Tilt_Z$ as a morphism that correspond to a morphism
\[ \stsh_{\mathbb{P}(V)}(a) \to v \otimes \stsh_{\mathbb{P}(V)}(a) \subset V \otimes \stsh_{\mathbb{P}(V)}(a) \subset \bigoplus_{k \geq 0} \Sym^k V \otimes_{\mathbb{C}} \stsh_{\mathbb{P}(V)}(a + k - 1) \]
via the adjunction.
This morphism $v : \stsh_Z(a) \to \stsh_Z(a-1)$ corresponds to an element
\[ v \otimes 1 \in V \otimes_{\mathbb{C}} \mathbb{C} \subset \bigoplus_{k \geq 0} \Sym^{k+1} V \otimes_{\mathbb{C}} \Sym^k V^* \]
via the isomorphism
\[ \Hom_Z(\stsh_Z(a), \stsh_Z(a-1)) \simeq \bigoplus_{k \geq 0} \Sym^{k+1} V \otimes_{\mathbb{C}} \Sym^{k} V^*. \]
We also define the action $f : \stsh_Z(a) \to \stsh_Z(a+1)$ of $f \in V^*$ on $\Tilt_Z$ as the morphism that is the pull-back of the morphism
\[ f : \stsh_{\mathbb{P}(V)}(a) \to \stsh_{\mathbb{P}(V)}(a+1) \]
by $\pi_Z : Z \to \mathbb{P}(V)$.
Note that this morphism corresponds to a morphism
\[ \stsh_{\mathbb{P}(V)}(a) \xrightarrow{f} \stsh_{\mathbb{P}(V)}(a+1) \subset \bigoplus_{k \geq 0} \Sym^k V \otimes_{\mathbb{C}} \stsh_{\mathbb{P}(V)}(a + k + 1) \]
via the adjunction, and also corresponds to an element
\[ 1 \otimes f \in  \mathbb{C} \otimes_{\mathbb{C}} V^* \subset \bigoplus_{k \geq 0} \Sym^k V \otimes_{\mathbb{C}} \Sym^{k+1} V^* \]
via the isomorphism
\[ \Hom_Z(\stsh_Z(a), \stsh_Z(a+1)) \simeq \bigoplus_{k \geq 0} \Sym^{k} V \otimes_{\mathbb{C}} \Sym^{k+1} V^*.\]
Now, it is clear that $v_1, \dots, v_N$ and $f_1, \dots, f_N$ generate $\Lambda_Z$ as a $S$-algebra and satisfy the commutative relation
\begin{align*}
v_i v_j &= v_j v_i \\
f_i f_j &= f_jf_i
\end{align*}
for any $i, j = 1, \dots, N$.

Next, we check that the relation
\[ f_iv_j = v_jf_i = x_{ij} \]
is satisfied.
By adjunction, the map
\[ f_iv_j : \stsh_Z(a) \to \stsh_Z(a) \]
corresponds to the composition
\begin{align*}
\stsh_{\mathbb{P}(V)}(a) &\to v_j \otimes \stsh_{\mathbb{P}(V)}(a) \subset \bigoplus_{k \geq 0} \Sym^k V \otimes_{\mathbb{C}} \stsh_{\mathbb{P}(V)}(a+k-1)  \\
&\xrightarrow{(\pi_Z)_*f_i} \bigoplus_{k \geq 0} \Sym^k V \otimes_{\mathbb{C}} \stsh_{\mathbb{P}(V)}(a+k), 
\end{align*}
where the map $(\pi_Z)_*f_i$ is the direct sum of the maps
\[ \Sym^k V \otimes_{\mathbb{C}} \stsh_{\mathbb{P}(V)}(a+k-1) \xrightarrow{\id \otimes f_i} \Sym^k V \otimes_{\mathbb{C}} \stsh_{\mathbb{P}(V)}(a+k). \]
Thus, this map factors through as
\[ \stsh_{\mathbb{P}(V)}(a) \to v_j \otimes \stsh_{\mathbb{P}(V)}(a) \xrightarrow{\id \otimes f_i} v_j \otimes \stsh_{\mathbb{P}(V)}(a+1) \subset \bigoplus_{k \geq 0} \Sym^k V \otimes_{\mathbb{C}} \stsh_{\mathbb{P}(V)}(a+k). \]
Similarly, the map
\[ v_j f_i : \stsh_Z(a) \to \stsh_Z(a) \]
corresponds to the composition
\begin{align*}
\stsh_{\mathbb{P}(V)}(a) \xrightarrow{f_i} \stsh_{\mathbb{P}(V)}(a+1) \subset \bigoplus_{k \geq 0} \Sym^k V \otimes_{\mathbb{C}} \stsh_{\mathbb{P}(V)}(a+k+1) \\
\xrightarrow{(\pi_Z)_*v_j} \bigoplus_{k \geq 0} \Sym^k V \otimes_{\mathbb{C}} \stsh_{\mathbb{P}(V)}(a+k)
\end{align*}
by adjunction, where the map $(\pi_Z)_*v_j$ is the direct sum of maps
\[ \Sym^k V \otimes_{\mathbb{C}} \stsh_{\mathbb{P}(V)}(a+k) \xrightarrow{v_j \otimes \id} \Sym^{k+1} V \otimes_{\mathbb{C}} \stsh_{\mathbb{P}(V)}(a+k). \]
Thus, this map factors through as
\[ \stsh_{\mathbb{P}(V)}(a) \xrightarrow{f_i} \stsh_{\mathbb{P}(V)}(a+1) \xrightarrow{v_j \otimes \id} v_j \otimes \stsh_{\mathbb{P}(V)}(a+1) \subset \bigoplus_{k \geq 0} \Sym^k V \otimes_{\mathbb{C}} \stsh_{\mathbb{P}(V)}(a+k). \]
Thus, $f_iv_j$ and $v_j f_i$ defines the same element in $\Hom_Z(\stsh_Z(a), \stsh_Z(a))$, and they correspond to an element
\[ x_{ij} =  v_j \otimes f_i \in V \otimes V^* \subset \bigoplus_{k \geq 0} \Sym^k V \otimes \Sym^k V^* (= \widetilde{R}) \]
via the isormorpshim
\[ \Hom_Z(\stsh_Z(a), \stsh_Z(a)) \simeq \bigoplus_{k \geq 0} \Sym^k V \otimes \Sym^k V^*. \]
Thus, we have the relation
\[ f_iv_j = v_jf_i = x_{ij}. \]
It is clear that $v_1, \dots, v_N$ and $f_1, \dots, f_N$ do not have other relations.
Therefore, we have the result.
\end{proof}

The following is one of main theorems in this paper.

\begin{thm} \label{NCCR quiver}
The non-commutative algebra $\Lambda_k$ is isomorphic to the path algebra $S\widetilde{\Gamma}$ of the double Beilinson quiver $\widetilde{\Gamma}$
with relations
\begin{align*}
v_iv_j &= v_jv_i ~~ \text{for all $1 \leq i, j \leq N$}, \\
f_if_j &= f_jf_i ~~ \text{for all $1 \leq i, j \leq N$}, \\
v_jf_i &= f_iv_j = x_{ij} ~~ \text{for all $1 \leq i, j \leq N$}, \\
\text{and} ~~~ &\sum_{i=1}^N f_iv_i = 0 = \sum_{i=1}^N v_if_i
\end{align*}
\end{thm}

\begin{proof}
By the exact sequence
\[ 0 \to \stsh_Z \to \stsh_Z \to \stsh_Y \to 0, \]
we have an exact sequence
\[ 0 \to \Lambda_Z \xrightarrow{\iota} \Lambda_Z \to \Lambda_k \to 0. \]
Note that the map $\iota : \Lambda_Z \to \Lambda_Z$ is given by the multiplication of $\sum_{i=1}^N x_{ii} = \sum_{i=1}^N v_i \otimes f_i \in S$.
Thus, the result follows from Theorem \ref{NCCR quiver Z}.
\end{proof}

\begin{rem} \rm
If we work over the base field $\mathbb{C}$ instead of $S$, we have
\[ \Lambda_k \simeq \mathbb{C} \widetilde{\Gamma}/J' \]
and $J'$ is an ideal that is generated by
\begin{align*}
v_iv_j = v_jv_i, ~ f_if_j &= f_jf_i, ~ v_jf_i = f_iv_j, \\
f_k v_j f_i = f_iv_j f_k, ~ &v_j f_i v_l = v_l f_i v_j\\
\sum_{i=1}^N f_iv_i = &0 = \sum_{i=1}^N v_if_i.
\end{align*}
The isomorphism $S\widetilde{\Gamma}/J \to \mathbb{C}\widetilde{\Gamma}/J'$ is given by
$v_i \mapsto v_i, ~ f_i \mapsto f_i, ~ x_{ij} \mapsto v_jf_i$. 
\end{rem}

\begin{eg} \rm
Let us consider the case $N = 2$.
In this case, the affine surface $\overline{B(1)}$ is given by
\[ \overline{B(1)} = \left\{ \left( \begin{array}{cc} a & b \\ c & -a \end{array} \right) \mid a^2 + bc = 0 \right\},  \]
and hence has a Du Val singularity of type $\mathrm{A}_1$ at the origin.
The resolution $Y = \lvert \Omega_{\mathbb{P}^1} \rvert \to \overline{B(1)}$ is the minimal resolution, and the NCCR $\Lambda_k$ is isomorphic to the smash product $\mathbb{C}[x,y] \sharp G$, where $G$ is a subgroup of $\SL_2$
\[ G = \left\{ \left( \begin{array}{cc} 1 & 0 \\ 0 & 1 \end{array} \right), \left( \begin{array}{cc} -1 & 0 \\ 0 & -1 \end{array} \right) \right\} \subset \SL_2. \]
 The quiver that gives the NCCR $\Lambda_k$ is given by
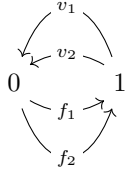
\[ \begin{tikzpicture}[
 description/.style={midway,
                     fill=white,
                     font=\scriptsize,
                     inner sep=2pt,
                     minimum height=.8em
 }]
 \matrix[matrix of math nodes, column sep=1cm]{
     |(n1)| 0 & |(n2)| 1 \\
 };
 \foreach \f/\t in {n1/n2}{ 
  \begin{scope}
   \coordinate (ulc) at ($(\f.south east)-(3pt,3pt)$);
   \coordinate (urc) at ($(ulc -| \t.west)+(3pt,0)$);
   \path[->,every node/.style={description}] 
       (\f.south east) edge[bend right=30] node (f1) {$f_1$} (\f.south east -| \t.west)
    (ulc) edge[bend right=70,looseness=2] node (fN) {$f_2$} (urc)
    ; 
   \end{scope}
  };
  \foreach \f/\t in {n2/n1}{ 
   \begin{scope}
   \coordinate (alc) at ($(\f.north west)+(3pt,3pt)$);
   \coordinate (arc) at ($(alc -| \t.east)+(-3pt,0)$);
   \path[->,every node/.style={description}]
    (\f.north west) edge[bend right=30] node (vN) {$v_2$} (\f.north west -| \t.east)
    (alc) edge[bend right=70,looseness=2] node (v1) {$v_1$} (arc)
    ;
  \end{scope}
 };
\end{tikzpicture} \]
and the relations (over $\mathbb{C}$) are given by $f_1v_1 + f_2v_2 = 0$, $v_1 f_1 + v_2 f_2 = 0$.

This quiver (with relations) coincides with the one that is described in Weyman and Zhao's paper \cite[Example 6.15]{WZ12}.
In \cite[Section 6]{WZ12}, Weyman and Zhao studied a description of an NCCR of a (maximal) determinantal variety of symmetric matrices as the path algebra of a quiver.
Since the surface $\overline{B(1)}$ is isomorphic to  a (maximal) determinantal variety of symmetric matrices
\[ \left\{ \left( \begin{array}{cc} a & b \\ b & c \end{array} \right) \mid ac - b^2 = 0 \right\}, \]
they obtained the above description of $\Lambda_k$ as a special case.
\end{eg}

%%%%%%%%%%%%%%%%%%%%%%%%%%%%%%%%%%%%%%%%%%%%%
\subsection{Remark: Alternative NCCRs of $\overline{B(1)}$} \label{subsec Rem NCCR}

The NCCR $\Lambda_k$ of $\overline{B(1)}$ that is constructed in the above subsection came from the Beilinson collection of $\mathbb{P}(V)$
\[ \D(\mathbb{P}(V)) = \langle \stsh, \stsh(1), \cdots, \stsh(N-1) \rangle. \]
In this subsection, we construct an NCCR of $R$ of another type from the different Beilinson collection
\[ \D(\mathbb{P}(V)) = \langle \Omega^{N-1}(N), \Omega^{N-2}(N-1), \cdots, \Omega^1(2), \stsh(1) \rangle. \]

\begin{defi} \rm
\begin{enumerate}
\item[(1)]  We define a representation $n_1 : \Stab_{\SL_N}(X_0) \to \SL_{N-1}$ as 
\[ \Stab_{\SL_N}(X_0) \ni \left( \begin{array}{c|ccc|c}
c & 0 & \cdots & 0 & 0 \\ \hline
 &  &  & & 0 \\
\ast &  &  A & & \vdots \\
 & & & & 0 \\ \hline
\ast &  & \ast &  & c
\end{array} \right) \mapsto \left( \begin{array}{c|ccc}
  c & 0 & \cdots & 0 \\ \hline
  &  & &   \\
 \ast & & A &  \\ 
  &  &  & 
\end{array} \right) \in \SL_{N-1}.
\]
For $0 \leq a \leq N-1$, we define a representation $n_a$ by
\[ n_a := \bigwedge^a n_1. \]
\item[(2)] Let $\mathcal{N}_a$ be a vector bundle on $B(1)$ that corresponds to the representation $n_a$.
\item[(3)] We set $N_a := H^0(B(1), \mathcal{N}_a)$. Then $N_a$ is a reflexive $R$-module.
\item[(4)] We define an $R$-algebra $\Lambda'$ by
\[ \Lambda' := \End_R\left(\bigoplus_{a=0}^{N-1} N_a\right). \]
\end{enumerate}
\end{defi}

As in Lemma \ref{thm1 lem}, we can relate the homogeneous vector bundle $\mathcal{N}_a$ with a (co)tangent bundle on a projective space.
We note that we have an isomorphism between vector bundles on $\mathbb{P}(V)$
\begin{align*}
\bigwedge^a (T_{\mathbb{P}(V)}(-1)) &\simeq (\Omega_{\mathbb{P}(V)}^a)^*(-a) \\
&\simeq \Omega_{\mathbb{P}(V)}^{N-a-1}(N) \otimes \stsh(-a) \\
&\simeq \Omega_{\mathbb{P}(V)}^{N-a-1}(N-a).
\end{align*}
Here, $T_{\mathbb{P}(V)}$ is the tangent bundle on $\mathbb{P}(V)$ and $\Omega_{\mathbb{P}(V)}$ is the cotangent bundle on $\mathbb{P}(V)$.

\begin{lem}
We have $\pi^*\Omega_{\mathbb{P}(V)}^{a-1}(a)|_U \simeq \mathcal{N}_{N-a}$.
\end{lem}

The proof is completely same as in Lemma \ref{thm1 lem}.

We want to show that the algebra $\Lambda'$ is an NCCR of $R$.
In order to show this, we need the following lemma.

\begin{lem}[\cite{BLV10}, Corollary 3.24] \label{BLV10 lem}
Let $\mathcal{M}_a^b(-c) := \lhom_{\mathbb{P}(V)}(\Omega_{\mathbb{P}(V)}^{b-1}(b), \Omega_{\mathbb{P}(V)}^{a-1}(a))(-c)$.
Then, the cohomology $H^d(\mathbb{P}(V), \mathcal{M}_a^b(-c))$ is not zero only in the following cases:
\begin{enumerate}
\item[(1)] If $d - c > 0$, then $d = 0$ and, necessarily, $c < 0$.
\item[(2)] If $d - c = 0$, then $c + b \in [\max\{a, b\}, \min\{N, a+b - 1\}]$.
\item[(3)] If $d - c= -1$, then $c -a \in [\max\{0, N-a-b-1\}, \min\{N-b, N-a\}]$.
\item[(4)] If $d - c < -1$, then $d = N - 1$, and necessarily, $c > N$.
\end{enumerate}
In particular, if $c \leq 0$, we have $H^d(\mathbb{P}(V), \mathcal{M}_a^b(-c)) = 0$ for all $d > 0$.
\end{lem}

From this lemma, we can obtain the following corollaries.

\begin{cor} \label{cor alt NCCR MCM}
For $0 \leq a \leq N-1$, we have $N _a \simeq \phi_*\pi^*\Omega_{\mathbb{P}(V)}^{N-a-1}(N-a)$ and $N_a$ is Cohen-Macaulay.
\end{cor}

\begin{proof} Let $k \geq 0$ be a non-negative integer.
Note that, $\Omega_{\mathbb{P}(V)}^{N-a-1}(N-a) \otimes \stsh(k) \simeq \mathcal{M}^{N}_{N-a}(k)$ and 
\begin{align*}
(\Omega_{\mathbb{P}(V)}^{N-a-1}(N-a))^* \otimes \stsh(k) &\simeq \Omega_{\mathbb{P}(V)}^{a}(N) \otimes \stsh(-N+a) \otimes \stsh(k) \\
&\simeq \Omega_{\mathbb{P}(V)}^{a}(a+1) \otimes \stsh(k-1) \\
&\simeq \mathcal{M}^1_{a+1}(k).
\end{align*}
Thus, by Lemma \ref{BLV10 lem} and Lemma \ref{lem 3-B}, we have
\[ H^i(Y, \pi^*\Omega_{\mathbb{P}(V)}^{N-a-1}(N-a)) = 0 = H^i(Y, \pi^*(\Omega_{\mathbb{P}(V)}^{N-a-1}(N-a))^*) \]
for $i > 0$, and hence by Lemma \ref{lem 3-A}, we have the $R$-module $\phi_*\pi^*\Omega_{\mathbb{P}(V)}^{N-a-1}(N-a)$ is (maximal) Cohen-Macaulay.
In particular, $\phi_*\pi^*\Omega_{\mathbb{P}(V)}^{N-a-1}(N-a)$ is reflexive and hence we have the desired isomorphism.
\end{proof}

\begin{cor} \label{tilting 2}
The bundle 
\[ \Tilt' := \bigoplus_{a = 1}^N \pi^*\Omega_{\mathbb{P}(V)}^{a-1}(a) \]
is a tilting bundle on $Y$ and there is an isomorphism as $R$-algebras
\[ \Lambda' \simeq \End_Y(\Tilt'). \]
In particular, the $R$-module $\bigoplus_{a=0}^{N-1} N_a$ gives an NCCR $\Lambda'$ of $R$.
\end{cor}

\begin{proof}
The bundle $(\Tilt')^* \otimes \Tilt'$ is the direct sum of $\pi^*\mathcal{M}_a^b(0)$.
By Lemma \ref{BLV10 lem} and Lemma \ref{lem 3-B}, we have 
\[ H^i(Y, \pi^*\mathcal{M}_a^b(0)) = 0 \]
for $i >0$ and hence we have
\[ \Ext_Y^i(\Tilt', \Tilt') = H^i(Y, (\Tilt')^* \otimes \Tilt') = 0 \]
for $i > 0$.
It is clear that the bundle generates the category $\D(\Qcoh(Y))$.
Therefore, the bundle $\Tilt'$ is tilting.
\end{proof}

\begin{cor}
Let us assume $N \geq 3$. In this case, although the two NCCRs $\Lambda_k, \Lambda'$ of $R$ are not isomorphic to each other, there is an equivalence of categories
\[ \D(Y) \simeq \D(\Lambda_k) \simeq \D(\Lambda'). \]
\end{cor}

\begin{proof}
The $R$-rank of the first NCCR $\Lambda_k$ is just $2N$ and the $R$-rank of the second NCCR $\Lambda'$ is
\[ 2\sum_{a=1}^N \rank \Omega^{a-1}_{\mathbb{P}(V)} = 2^{N}. \]
Thus, if $N \geq 3$, $\Lambda_k$ and $\Lambda'$ are not isomorphic to each other but have the equivalent derived categories, where the equivalence is given by the composition
\[ \D(\Lambda') \xrightarrow{- \otimes_{\Lambda'} \Tilt'} \D(Y) \xrightarrow{\RHom_Y(\Tilt_k, - )} \D(\Lambda_k). \]
This shows the result.
\end{proof}

At the end of this subsection, we give another type of tilting bundles that we use in the later section (Section \ref{subsubsect IW-multi}).

\begin{prop} \label{tilting 3}
The vector bundle
\[ \mathcal{S}_k = \bigoplus_{a = -N+2}^0 \stsh_Y(a) \oplus \left( \pi^* \Omega_{\PP(V)}^k \otimes \stsh_Y(1) \right) \]
and its dual vector bundle $\mathcal{S}_k^*$ are tilting bundle on $Y$ for all $0 \leq k \leq N -1$.
\end{prop}

\begin{proof}
As in Lemma \ref{tilting 2}, the claim follows from direct computations using Lemma \ref{BLV10 lem}.
\end{proof}

%%%%%%%%%%%%%%%%%%%%%%%%%%%%%%%%%%%%%%%%%%%%%%%%%%%%%%%%%%%%%%%%%%%%%%%
\section{From an NCCR to crepant resolutions} \label{section, moduli}

\subsection{Main theorem}

In this section, we recover the crepant resolutions $Y$ and $Y^+$ of $\overline{B(1)}$ from the NCCR $\Lambda_k$.
Again, let $\widetilde{\Gamma}$ be the double Beilinson quiver

\[ \begin{tikzpicture}[%
 % 矢印の上に乗っけるラベルの書式の設定
 description/.style={midway,              % 中間に表示
                     fill=white,          % 矢印の中間を白で塗り潰し
                     font=\scriptsize,    % 添え字と同じ大きさで表示
                     inner sep=2pt,       % 矢印と文字は 2pt 感覚を空ける
                     minimum height=.8em  % ノードの高さは適当に0.8emに揃える（中黒点の数を三つに合わせるため）
 }]
 \matrix[matrix of math nodes, column sep=1cm]{ % 数式のノードを横幅 1cm の間隔で並べる
     % (…) 内は数式の名前
   |(n1)| 0 & |(n2)| 1 & |(dots)| \cdots & |(N1)| N - 2 & |(N)| N  -1 \\
 };
 % (1 と 2), (2 と …), (… と N - 1), (N-1 と N) の間の上下に矢印を延ばす
 \foreach \f/\t in {n1/n2, n2/dots, dots/N1, N1/N}{ 
  \begin{scope}
   % 定義域 (\f) から値域 (\t)への 下部矢印の描画
   % f_N の矢印の始点 (ulc) は、定義域の右下から 3pt ずつ左下にズラして設定
   \coordinate (ulc) at ($(\f.south east)-(3pt,3pt)$);
   % 行き先 (urc) は ulc を行き先の左端辺に射影して、更に右に3ptズラして対称になるようにする。
   \coordinate (urc) at ($(ulc -| \t.west)+(3pt,0)$);
   \path[->,every node/.style={description}] % 矢印を、上の設定したラベルの書式に従って描くよ！
    % 定義域の右下点から、それを値域の左辺に射影した点まで f_1 の矢印を延ばす。
       % ラベルのある座標に (f1) という名前をつける。
    (\f.south east) edge[bend right=30] node (f1) {$f_1$} (\f.south east -| \t.west)
    % 上で設定した (ulc) から (urc) に向けて f_N の矢印を延ばす。
    (ulc) edge[bend right=70,looseness=2] node (fN) {$f_N$} (urc)
    ; % 矢印の設定はここまで。
   %
   % (f1) の下端と (fN) の上端を、太めの点々で結ぶ。
   \path[draw=black, dotted, thick] (f1.south) to (fN.north);
   \end{scope}
  };
   % 上部矢印 v_1, v_N の描画。これまでと symmetric にやっているだけ。
  \foreach \f/\t in {n2/n1, dots/n2, N1/dots, N/N1}{ 
   \begin{scope}
   \coordinate (alc) at ($(\f.north west)+(3pt,3pt)$);
   \coordinate (arc) at ($(alc -| \t.east)+(-3pt,0)$);
   \path[->,every node/.style={description}]
    (\f.north west) edge[bend right=30] node (vN) {$v_N$} (\f.north west -| \t.east)
    (alc) edge[bend right=70,looseness=2] node (v1) {$v_1$} (arc)
    ;
   \path[draw=black, dotted, thick] (vN.north) to (v1.south);
  \end{scope}
 };
\end{tikzpicture} \]

with relations 
\begin{align*}
v_iv_j &= v_jv_i ~~ \text{for all $1 \leq i, j \leq N$}, \\
f_if_j &= f_jf_i ~~ \text{for all $1 \leq i, j \leq N$}, \\
v_jf_i &= f_iv_j = x_{ij} ~~ \text{for all $1 \leq i, j \leq N$}, \\
\text{and} ~ ~~ &\sum_{i=1}^N f_iv_i = 0 = \sum_{i=1}^N v_if_i.
\end{align*}

For a commutative $\mathbb{C}$-algebra $A$, let $\Rep(A)$ the set of representations $W$ of the quiver $\widetilde{\Gamma}$ (with the above relations)
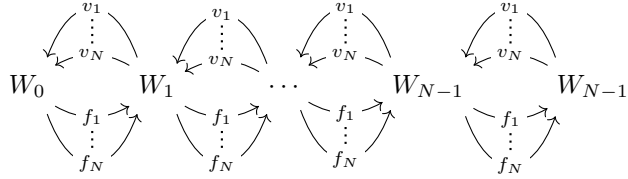
\[ \begin{tikzpicture}[%
 % 矢印の上に乗っけるラベルの書式の設定
 description/.style={midway,              % 中間に表示
                     fill=white,          % 矢印の中間を白で塗り潰し
                     font=\scriptsize,    % 添え字と同じ大きさで表示
                     inner sep=2pt,       % 矢印と文字は 2pt 感覚を空ける
                     minimum height=.8em  % ノードの高さは適当に0.8emに揃える（中黒点の数を三つに合わせるため）
 }]
 \matrix[matrix of math nodes, column sep=1cm]{ % 数式のノードを横幅 1cm の間隔で並べる
     % (…) 内は数式の名前
   |(n1)| W_0 & |(n2)| W_1 & |(dots)| \cdots & |(N1)| W_{N-1} & |(N)| W_{N - 1} \\
 };
 % (1 と 2), (2 と …), (… と N - 1), (N-1 と N) の間の上下に矢印を延ばす
 \foreach \f/\t in {n1/n2, n2/dots, dots/N1, N1/N}{ 
  \begin{scope}
   % 定義域 (\f) から値域 (\t)への 下部矢印の描画
   % f_N の矢印の始点 (ulc) は、定義域の右下から 3pt ずつ左下にズラして設定
   \coordinate (ulc) at ($(\f.south east)-(3pt,3pt)$);
   % 行き先 (urc) は ulc を行き先の左端辺に射影して、更に右に3ptズラして対称になるようにする。
   \coordinate (urc) at ($(ulc -| \t.west)+(3pt,0)$);
   \path[->,every node/.style={description}] % 矢印を、上の設定したラベルの書式に従って描くよ！
    % 定義域の右下点から、それを値域の左辺に射影した点まで f_1 の矢印を延ばす。
       % ラベルのある座標に (f1) という名前をつける。
    (\f.south east) edge[bend right=30] node (f1) {$f_1$} (\f.south east -| \t.west)
    % 上で設定した (ulc) から (urc) に向けて f_N の矢印を延ばす。
    (ulc) edge[bend right=70,looseness=2] node (fN) {$f_N$} (urc)
    ; % 矢印の設定はここまで。
   %
   % (f1) の下端と (fN) の上端を、太めの点々で結ぶ。
   \path[draw=black, dotted, thick] (f1.south) to (fN.north);
   \end{scope}
  };
   % 上部矢印 v_1, v_N の描画。これまでと symmetric にやっているだけ。
  \foreach \f/\t in {n2/n1, dots/n2, N1/dots, N/N1}{ 
   \begin{scope}
   \coordinate (alc) at ($(\f.north west)+(3pt,3pt)$);
   \coordinate (arc) at ($(alc -| \t.east)+(-3pt,0)$);
   \path[->,every node/.style={description}]
    (\f.north west) edge[bend right=30] node (vN) {$v_N$} (\f.north west -| \t.east)
    (alc) edge[bend right=70,looseness=2] node (v1) {$v_1$} (arc)
    ;
   \path[draw=black, dotted, thick] (vN.north) to (v1.south);
  \end{scope}
 };
\end{tikzpicture} \]
such that, for each $i$, $W_i$ is a (constant) rank $1$ projective $A$-module and $W$ is generated by $W_0 = A$.

The goal of this section is to show the following theorem.

\begin{thm}[cf. \cite{VdB04b}, Section 6] \label{NCCR to CR}
$Y$ is the fine moduli space of the functor $\Rep$.
The universal bundle is $\Tilt_{N-1}$.
\end{thm}

Recall that the NCCR $\Lambda_k$ is isomorphic to the path algebra $S\widetilde{\Gamma}/J$ where $J$ is the ideal generated by the above relations.
Therefore, Theorem \ref{NCCR to CR} means that we can recover a crepant resolution $Y$ of $\overline{B(1)}$ (and a tilting bundle on $Y$) from the NCCR $\Lambda_k$ as a moduli space of $\Lambda_k$-modules (and its universal bundle).
The other crepant resolution $Y^+$ is also recovered as the fine moduli space of another functor $\Rep^+$
(see Remark \ref{rem NCCR to CR}).

\subsection{Projective module of rank $1$}
Let $A$ be a (commutative, noetherian) $\mathbb{C}$-algebra.
In this subsection, we recall some basic properties of projective $A$-modules of (constant) rank $1$.
First, we recall the following fundamental result for projective modules.
One can find the following proposition in Chapter II, \S 5, 2, Theorem 1 of \cite{Bourbaki}.

\begin{prop} \label{Bou}
Let $M$ be a finitely generated $A$-module. Then, the following are equivalent.
\begin{enumerate}
\item[(i)] $M$ is projective.
\item[(ii)] For all $\mathfrak{p} \in \Spec A$, there exists a non-negative integer $r(\mathfrak{p}) \in \mathbb{Z}_{\geq 0}$ such that $M_{\mathfrak{p}} \simeq A^{r(\mathfrak{p})}_{\mathfrak{p}}$.
\item[(iii)] There exist $f_1, \dots, f_r \in A$ such that they generate the unite ideal of $A$ and $M_{f_i}$ is a free $A_{f_i}$-module for each $i$.
\end{enumerate}
\end{prop}

From this proposition, we have the following.

\begin{cor}
Let $M$ be a finitely generated $A$-module.
Then, the following are equivalent.
\begin{enumerate}
\item[(1)] The sheaf on $\Spec A$ that associates to $M$ is an invertible sheaf.
\item[(2)] $M$ is a projective $A$-module of constant rank $1$.
\end{enumerate}
\end{cor}

Thus, if we consider projective modules of constant rank $1$, the symmetric product of them coincides with the tensor product.

\begin{lem}
Let $P$ be a (finitely generated) projective $A$-module of (constant) rank $1$.
Let $\mathfrak{S}_k$ be a group of permutations of the set $\{1, 2, \dots, k \}$.
Then, for any $m_1, m_2,  \dots, m_k \in P$ and any $\sigma \in \mathfrak{S}_k$, we have
\[ m_1 \otimes m_2 \otimes \cdots \otimes m_k = m_{\sigma(1)} \otimes m_{\sigma(2)} \otimes \cdots \otimes m_{\sigma(k)} \]
in $P^{\otimes k}$.
In particular, we have
\[ P^{\otimes k} \simeq \Sym_A^k P \]
as an $A$-module.
\end{lem}

\begin{proof}
This is the direct consequence of Proposition \ref{Bou} (iii) and the gluing property of sheaves.
\end{proof}

\begin{cor} \label{cor proj mod}
Let $P$ be a (finitely generated) projective $A$-module of (constant) rank $1$.
For any $u \in P^{\vee} = \Hom_A(P, A)$ and $m_1 \dots m_k \in \Sym_A^k P$, we have
\[ u(m_i) \cdot m_1 \cdots \widehat{m_i} \cdots m_j \cdots m_k = u(m_j) \cdot m_1 \cdots m_i \cdots \widehat{m_j} \cdots m_k \]
in $\Sym_A^{k-1} P$ for all $1 \leq i < j \leq k$.
In particular, the map
\[ \Sym_A^k P \to \Sym_A^{k-1} P, ~~~~ m_1 \dots m_k \mapsto u(m_i) \cdot m_1 \cdots \widehat{m_i} \cdots m_k \]
is well-defined and does not depend on the choice of $i$.
\end{cor}

Corollary \ref{cor proj mod} will be used in Section \ref{subsect NCCR to CR proof} to construct a representation of $\widetilde{\Gamma}$ from a projective module $P$ of constant rank $1$.

\subsection{An easy case}

In order to prove Theorem \ref{NCCR to CR}, we first study an easier functor $\mathcal{R}$.
For commutative $\mathbb{C}$-algebra $A$, let $\mathcal{R}(A)$ be the set of representation $W$ of Beilinson quiver $\Gamma$
\[ \begin{tikzcd}
   W_0 \arrow[r, shift left=1, phantom, "\vdots" description] \arrow[r, shift left=5, "f_1" description] \arrow[r, shift right=5, "f_N" description] &
   W_1 \arrow[r, shift left=1, phantom, "\vdots" description] \arrow[r, shift left=5, "f_1" description] \arrow[r, shift right=5, "f_N" description] &
   \cdots \arrow[r, shift left=1, phantom, "\vdots" description] \arrow[r, shift left=5, "f_1" description] \arrow[r, shift right=5, "f_N" description] &
   W_{N-2} \arrow[r, shift left=1, phantom, "\vdots" description] \arrow[r, shift left=5, "f_1" description] \arrow[r, shift right=5, "f_N" description] &
   W_{N-1}
\end{tikzcd} \]
with usual relations
\[ f_i f_j = f_j f_i ~~(i, j = 1, \dots, N) \]
such that each $W_i$ is rank $1$ projective $A$-module and $W$ is generated by $W_0 = A$.

Let us consider a rank $1$ projective $A$-module $P$ and split injective morphism $\alpha : P \to V \otimes_{\mathbb{C}} A$.
For the pair $(P, \alpha)$, we define a representation $W_{\alpha}$ of $\Gamma$ as follows.
Let $(W_{\alpha})_k := \Sym_A^k P^{\vee}$ where $P^{\vee} := \Hom_A(P, A)$ is the dual of $P$.
The action of $f \in V$ is defined by
\[ f : \Sym_A^k P^{\vee} \to \Sym_A^{k+1} P^{\vee}, ~~ u^1 \cdots u^k \mapsto \alpha^{\vee}(f)u^1 \cdots u^k. \]
By construction, we have $W_{\alpha} \in \mathcal{R}(A)$.

\begin{prop} \label{easy prop}
For any $W \in \mathcal{R}(A)$, there exists a unique pair $(P, \alpha)$ as above such that $W \simeq W_{\alpha}$.
\end{prop}

\begin{proof}
Let $W \in \mathcal{R}(A)$. Since $W$ is generated by the first component $W_0 = A$, we have a surjective morphism
\[ \pi : V^* \otimes_{\mathbb{C}} A \to W_1. \]
Since $W_1$ is a projective $A$-module, the morphism $\pi$ is split surjection.
If $W = W_{\alpha}$ for some $(P, \alpha)$, then we have $P = W_1^{\vee} = \Hom_A(W_1, A)$ and $\alpha = \pi^{\vee} = \Hom_A(\alpha, -)$.
This shows the uniqueness of $(P, \alpha)$.

For arbitrary $W$, since $W$ is generated by $W_0$, $W$ is a quotient of a $A$-module
\[ \bigoplus_{i=0}^{N-1} \Sym_A^i (V^* \otimes_{\mathbb{C}} A/\Ker \pi) \simeq \bigoplus_{i=0}^{N-1} \Sym_A^i P^{\vee}. \]
However, $W$ and $\bigoplus_{i=0}^{N-1} \Sym_A^i P^{\vee}$ have the same $A$-rank $N-1$, we have
\[ W \simeq \bigoplus_{i=0}^{N-1} \Sym_A^i P^{\vee}. \]
This shows the lemma.
\end{proof}

Thus, we have the next result.

\begin{cor}
The functor $\mathcal{R}$ is represented by the projective space $\mathbb{P}(V)$ and the universal sheaf is $\bigoplus_{a=0}^{N-1} \stsh(a)$.
\end{cor}

In the next subsection, we prove Theorem \ref{NCCR to CR} by using Proposition \ref{easy prop}.

\subsection{Proof of Theorem \ref{NCCR to CR}} \label{subsect NCCR to CR proof}

Let us consider a projective $A$-module $P$ of rank $1$ and a pair of morphisms $(\alpha, \beta)$, where
\begin{align*}
\alpha : P  \hookrightarrow V \otimes_{\mathbb{C}} A \\
\beta : P^{\vee} \to V^* \otimes_{\mathbb{C}} A
\end{align*}
that satisfies $\beta^{\vee} \circ \alpha = 0$ (equivalently, $\alpha^{\vee} \circ \beta = 0$) and $\alpha$ is injective and split.
We note that the triple $(P, \alpha, \beta^{\vee})$ is a (stable) representation of Nakajima's quiver of type A over the commutative algebra $A$.
Via the basis $v_1, \dots, v_N$ of $V$, we set the matrix
\[ (a_{ij}) := \alpha \circ \beta^{\vee} : V \otimes_{\mathbb{C}} A \to V \otimes_{\mathbb{C}} A. \]

For a triple $(P, \alpha, \beta)$ as above, we define a representation $W_{\alpha\beta}$ as follows.
We set $(W_{\alpha\beta})_a := \Sym_A^a P^{\vee}$.
The action of $f \in V^*$ is given by
\[ f : \Sym_A^a P^{\vee} \to \Sym_A^{a+1} P^{\vee}, ~~ u^1 \cdots u^a \mapsto \alpha^{\vee}(f)u^1 \cdots u^a. \]
The action of $v \in V$ is given by
\[ v : \Sym_A^a P^{\vee} \to \Sym_A^{a-1} P^{\vee}, ~~ u^1 \cdots u^a \mapsto u^j(\beta^{\vee}(v)) \cdot u^1 \cdots \widehat{u^j} \cdots u^a. \]
This map is well-defined and does not depends the choice of $j$ by Corollary \ref{cor proj mod}.

First, we need to check the following

\begin{lem}
For a triple $(P, \alpha, \beta)$ as above, we have $W_{\alpha\beta} \in \Rep(A)$.
\end{lem}

\begin{proof}
We need to check the following.
\begin{enumerate}
\item[(1)] $v_i v_j = v_j v_i$ and $f_i f_j = f_j f_i$.
\item[(2)] $v_jf_i = f_i v_ j = a_{ij}$.
\item[(3)] $\sum_{i=1}^N f_i v_i = 0 = \sum_{i=1}^N v_i f_i$.
\item[(4)] $W_{\alpha\beta}$ is generated by $(W_{\alpha\beta})_0$.
\end{enumerate}
(1) and (4) trivially follows from the construction of $W_{\alpha\beta}$.
We need to check (2) and (3).
First, we check (2).
The action on $f_i v_j$ on $(W_{\alpha\beta})_k = \Sym_A^k P^{\vee}$ is given by
\begin{align*}
f_iv_j(u^1 \cdots u^k) = u^l(\beta^{\vee}(v_j)) \cdot \alpha^{\vee}(f_i)u^1 \cdots \widehat{u^l} \cdots u^k,
\end{align*}
for some $l$.
On the other hand, $v_j f_i$ acts on $(W_{\alpha\beta})_k$ by
\begin{align*}
& v_j f_i (u^1 \cdots u^k) \\
= & v_j(\alpha^{\vee}(f_i)u^1 \dots u^k) \\
= & u^l(\beta^{\vee}(v_j)) \cdot \alpha^{\vee}(f_i)u^1 \cdots \widehat{u^l} \cdots u^k \\
= &f_iv_i(u^1 \cdots u^k).
\end{align*}
We note that we also have
\[ v_j f_i (u^1 \cdots u^k) = \alpha^{\vee}(f_i)(\beta^{\vee}(v_j)) \cdot u^1 \cdots u^k \]
and $\alpha^{\vee}(f_i)(\beta^{\vee}(v_j)) = f_i((\alpha \circ \beta^{\vee})(v_j)) = a_{ij} \in A$.
Hence we have
\[ (v_jf_i)(u^1 \cdots u^k) = (f_iv_j)(u^1 \cdots u^k) = a_{ij} \cdot u^1 \cdots u^k. \]
This shows (2).
Next, we check (3).
From the above computation, we have
\[ (\sum_{i = 1}^N f_iv_i)(u^1 \cdots u^k) =  \left( \sum_{i=1}^N u^l(\beta^{\vee}(v_i)) \cdot \alpha^{\vee}(f_i) \right) \cdot u^1 \cdots \widehat{u^l} \cdots u^k. \]
Thus, we have to show that
\[ \sum_{i=1}^N u^l(\beta^{\vee}(v_i)) \cdot \alpha^{\vee}(f_i) = 0. \]
Let us consider the composition
\[ P^{\vee} \xrightarrow{\beta}  V^* \otimes_{\mathbb{C}} A \xrightarrow{\alpha^{\vee}} P^{\vee}. \]
Note that $\beta(u^l) = \sum_{i=1}^N (\beta(u^l))(v_i) \cdot f_i = \sum_{i=1}^N u^l(\beta^{\vee}(v_i)) \cdot f_i$.
Hence we have $\sum_{i=1}^N u^l(\beta^{\vee}(v_i)) \cdot \alpha^{\vee}(f_i) = (\alpha^{\vee} \circ \beta)(u^l) = 0$.
The same argument shows that we have
\[ \sum_{i=1}^N v_if_i = 0. \]
This shows (3).
\end{proof}

Next, we show the next proposition.

\begin{prop}
For any $W \in \Rep(A)$, there exists a unique $(P, \alpha, \beta)$ as above such that $W \simeq W_{\alpha\beta}$.
\end{prop}

\begin{proof}
By forgetting the action of $V$, we can regard $W$ as an object in $\mathcal{R}$.
Thus, by Proposition \ref{easy prop}, there exist a projective $A$-module $P$ and a split injective morphism $\alpha : P \to V \otimes_{\mathbb{C}} A$
such that $W \simeq W_{\alpha}$.
We want to construct the morphism $\beta : P^{\vee} \to V^* \otimes_{\mathbb{C}} A$.

The action of $v_i \in V$ on $W_1 = P^{\vee}$
\[ v_i : P^{\vee} \to W_0 = A \]
is an element in $\Hom_A(P^{\vee}, A) \simeq P$.
Let $p_i \in P$ an element in $P$ that corresponds to $v_i \in V$ via the above isomorphism.
By using this, we set a morphism
\[ \gamma : V \otimes_{\mathbb{C}} A \to P \]
by
\[ v_i \otimes 1 \mapsto p_i, \]
and we set $\beta := \gamma^{\vee}$.
In order to complete the proof, we need to check the next two properties.
\begin{enumerate}
\item[(1)] $\alpha^{\vee} \circ \beta = 0$,
\item[(2)] The given action of $V$ on $W$ coincides with the one that is determined by $\beta$.
\end{enumerate}
First, we check (1).
For $u \in P^{\vee}$, we have
\[ \beta(u) = \sum_{i=1}^N (\beta(u))(v_i) \cdot f_i. \]
Therefore, we have
\begin{align*}
(\alpha^{\vee} \circ \beta)(u) = \sum_{i=1}^N (\beta(u))(v_i) \cdot \alpha^{\vee}(f_i)
= \sum_{i=1}^N f_i((\beta(u))(v_i))
= \sum_{i=1}^N (f_iv_i)(u) 
= 0.
\end{align*}
The last equality follows from the relation $\sum_{i=1}^N f_iv_i = 0$.
This shows (1).
Next, we check (2).
We show that the action
\[ v_i : W_k \to W_{k-1} \]
coincides with the desired one by induction on $k$.
For $k = 1$, this is true by the construction of $\beta$.
Let us assume $k > 1$.
By definition, we have
\[ (v_if_j)(u^1 \cdots u^k) = v_j(\alpha^{\vee}(f_j)u^1 \cdots u^k). \]
On the other hand, by the relation and the induction hypothesis, we have
\begin{align*}
(v_if_j)(u^1 \cdots u^k)
= & a_{ji} \cdot u^1 \cdots u^k  \\
= & \alpha^{\vee}(f_j)(\beta^{\vee}(v_i)) \cdot u^1 \cdots u^k
\end{align*}
Since $\alpha^{\vee} : V^* \otimes_{\mathbb{C}} A \to P^{\vee}$ is surjective, we can replace $\alpha^{\vee}(f_j)$ in the above equation by arbitrary $u \in P^{\vee}$, 
and hence we have
\[ v_j(uu^1 \cdots u^k) = u(\beta^{\vee}(v_i)) \cdot u^1 \cdots u^k. \]
This shows (2) and the proof is completed.
\end{proof}

The triple $(P, \alpha, \beta^{\vee})$ gives a representation of the Nakajima's quiver $\overline{Q^{\heartsuit}}$ over $A$
\[ \begin{tikzpicture}[
 description/.style={midway,
                     fill=white,
                     font=\scriptsize,
                     inner sep=2pt,
                     minimum height=.8em
 }]
 \matrix[matrix of math nodes, column sep=1cm]{
     |(n1)| P & |(n2)| V \otimes_{\mathbb{C}} A \\
 };
 \foreach \f/\t in {n1/n2}{ 
  \begin{scope}
  % \coordinate (ulc) at ($(\f.south east)-(3pt,3pt)$);
   %\coordinate (urc) at ($(ulc -| \t.west)+(3pt,0)$);
   \path[->,every node/.style={description}] 
       (\f.south east) edge[bend right=30] node (f1) {$\alpha$} (\f.south east -| \t.west)
    %(ulc) edge[bend right=70,looseness=2] node (fN) {$f_2$} (urc)
    ; 
   \end{scope}
  };
  \foreach \f/\t in {n2/n1}{ 
   \begin{scope}
   %\coordinate (alc) at ($(\f.north west)+(3pt,3pt)$);
   %\coordinate (arc) at ($(alc -| \t.east)+(-3pt,0)$);
   \path[->,every node/.style={description}]
    (\f.north west) edge[bend right=30] node (vN) {$\beta^{\vee}$} (\f.north west -| \t.east)
    %(alc) edge[bend right=70,looseness=2] node (v1) {$v_1$} (arc)
    ;
  \end{scope}
 };
\end{tikzpicture} \]
of dimension vector $(1, N)$, where $Q$ is the $\mathrm{A}_1$ quiver (i.e. a point).
As it was explained above, the variety $Y$ is given by
\[ Y = \{ (L, X) \in \mathbb{P}(V) \times \End_{\mathbb{C}}(V) \mid X(V) \subset L, X(L) = 0 \}. \]
This is a description of $Y$ as the Nakajima's quiver variety of type $\mathrm{A}_1$ with dimension vector $(1, N)$.
From this presentation of $Y$, we find that $Y$ represents the functor $\Rep$.
Moreover, since Nakajima's quiver varieties admit a natural symplectic structure,
we can say that a symplectic structure of $Y$ can be recovered from the NCCR as well.

For the details of Nakajima's quiver variety and the notation that we used above, see \cite{Gi09}.

\begin{rem} \label{rem NCCR to CR} \rm
Let $\Rep^+(A)$ be a set consists of the representations of $\widetilde{\Gamma}$ with the relations in Theorem \ref{NCCR quiver} of dimension vector $(1, 1, \dots, 1)$ and generated by the last component $W_{N-1}$.
Then, the dual argument shows that the functor $\Rep^+$ represented by the variety $Y^+$.
\end{rem}

\subsection{Simple representations}

In the rest of this section, we determine simple representations that are contained in $\Rep(\mathbb{C})$.

\begin{lem} \label{lem simple}
A representation $W = (W_k)_k \in \Rep(\mathbb{C})$ is simple if and only if it is generated by the last component $W_{N-1}$.
\end{lem}

\begin{proof}
If $W$ is not generated by $W_{N-1}$, 
the subrepresentation $W'$ that is generated by $W_{N-1}$ defines a non-trivial subrepresentation of $W$, and hence $W$ is not simple.

On the other hand, let $W' = (W'_k)_k$ be a non-zero subrepresentation of $W$.
Then, the last part of subrepresentation $W'_{N-1}$ coincides with the one $W_{N-1}$ of $W$.
Indeed, since $W'$ is non-zero, there exists $k$ such that $W'_k = \Sym_{\mathbb{C}}^k P$, where $P$ is a one-dimensional vector space over $\mathbb{C}$.
As the map $\alpha^{\vee} : V^* \to P$ is surjective, there exists $f \in V$ such that the image of the map
\[ \alpha^{\vee}(f)^{N-k-1} : \Sym_{\mathbb{C}}^k P \to \Sym_{\mathbb{C}}^{N-1} P \]
is non-zero.
Therefore, we have $W'_{N-1} \neq 0$ and hence we have $W'_{N-1} = W_{N-1}$.
Thus, if $W$ is generated by the last component $W_{N-1}$, the subrepresentation $W'$ should be $W$ itself.
\end{proof}

\begin{cor}
A representation $W = (W_k)_k \in \Rep(\mathbb{C})$ is simple if and only if the map $\beta : P^{\vee} \to V$ is injective.
\end{cor}

\begin{proof}
Let $W$ be a simple representation.
Then, by Lemma \ref{lem simple}, $W$ is generated by the last part $W_{N-1}$.
Thus, for at least one $i$, the map $v_i : W_1 = P^{\vee} \to W_0 = \mathbb{C}$ is non-zero.
Therefore, if we set an element $p_i \in P$ that corresponds $v_i$ via the identification $P \simeq \Hom_{\mathbb{C}}(P^{\vee}, \mathbb{C})$, the map $\gamma : V \to P, ~~ v_i \mapsto p_i$ is non-zero and hence surjective.
Recall that the morphism $\beta : P \to V$ is defined as the dual map of $\gamma$.
Thus, we have that the map $\beta$ is injective.

On the other hand, if $\beta$ is injective, we have that the representation $W$ is generated by $W_{N-1}$ from the construction.
\end{proof}

Let $W \in \Rep(\mathbb{C})$ and $(P, \alpha, \beta)$ a triple that defines $W$.
Then, $\alpha(P) \subset V$ defines a line in $V$ and the composition $\alpha \circ \beta^{\vee}$ defines an element in $\End_{\mathbb{C}}(V)$.
Moreover, a pair $(\alpha(P), \alpha \circ \beta^{\vee}) \in \mathbb{P}(V) \times \End_{\mathbb{C}}(V)$ defines a point of $Y$ that corresponds to $W$ via the identification $\Rep(\mathbb{C}) \simeq Y(\mathbb{C})$.

If $\beta$ is not injective, $\beta$ must be zero and hence the corresponding point of  $Y$ belongs to the zero section 
\[ j(\PP(V)) = \{(L, 0) \in \PP(V) \times \End_{\mathbb{C}}(V)  \}. \]
Conversely, if the point $(\alpha(P), \alpha \circ \beta^{\vee}) \in Y$ lies on the zero section, the map $\beta$ must be zero and hence not injective.

By summarizing the above discussion, we have the following theorem.

\begin{thm} \label{thm simple}
Let $W$ be a representation that belongs to the set $\Rep(\mathbb{C})$.
Then, the following are equivalent.
\begin{enumerate}
\item[(1)] $W$ is simple.
\item[(2)] $W$ is generated by the last component $W_{N-1}$.
\item[(3)] $W$ is corresponds to a point of $Y$ that lies over the non-singular part of $\overline{B(1)}$ via the identification $\Rep(\mathbb{C}) \simeq Y(\mathbb{C})$.
\end{enumerate}
\end{thm}

Of course, the corresponding argument holds for $Y^+$ and $\Rep^+$.

%%%%%%%%%%%%%%%%%%%%%%%%%%%%%%%%%%%%%%%%%%%%%%%%%%%%%%%%%%%%%%%%%%%%%%%
\section{Kawamata-Namikawa's equivalence for Mukai flops and P-twists} \label{section, KN}

In this section, we always assume $N \geq 3$.

\subsection{Kawamata-Namikawa's equivalence} \label{subsect KN}

%First of all, we recall notations that we adopt in this section.

%Let us consider the case $r = 1$. 
%Then, $Y = Y(1)$ is the total space of the cotangent bundle $\Omega_P^1$ of a projective space $P = \mathbb{P}(V)$.
%\begin{nota} \rm

Recall that the map $\phi : Y \to \overline{B(1)}$ contracts the zero section $j : {\mathbb{P}(V)} \hookrightarrow Y$ to $0 \in \overline{B(1)}$.
This is a flopping contraction and the flop is $Y^+ =  \lvert \Omega_{\mathbb{P}(V^*)} \rvert \xrightarrow{\phi^+} \overline{B(1)}$, where $\mathbb{P}(V^*)$ is the dual projective space of $\mathbb{P}(V)$.
In the following, we write $\PP := \mathbb{P}(V)$ and $\mathbb{P}^{\vee} := \PP(V^*)$ for short.
\[ \begin{tikzcd}
\PP \arrow[r, hook, "j"] & Y \arrow[rd, "\phi"] & & Y^+ \arrow[ld, "\phi^+"'] & \PP^{\vee} \arrow[l, hook', "j'"'] \\
 & & \overline{B(1)} & &
 \end{tikzcd} \]
As in the above sections, let $\pi : Y \to \mathbb{P}$ and $\pi' : Y^+ \to \mathbb{P}^{\vee}$ be the projections, 
and we set $\stsh_Y(1) := \pi^*\stsh_{\PP}(1)$ and $\stsh_{Y^+}(1) := (\pi')^*\stsh_{\PP^{\vee}}(1)$.
Then, the vector bundles
\begin{align*}
\Tilt_k := \bigoplus_{a=-N+k+1}^k \stsh_Y(a), \\
\Tilt^+_k := \bigoplus_{a = - N + k + 1}^k \stsh_{Y^+}(a)
\end{align*}
on $Y$, $Y^+$, respectively, are tilting bundles.
Moreover, we have an $R$-algebra isomorphism
\[ \Lambda_k := \End_Y(\Tilt_k) \simeq \End_{Y^+}(\Tilt^+_{N-k-1}), \]
by Theorem \ref{main thm flop side} (4).
%\end{nota}

By using the above tilting bundles, we have equivalences of categories
\begin{align*}
\Psi_k &:= \RHom_Y(\Tilt_k, - ) : \D(Y) \xrightarrow{\sim} \D(\modu(\Lambda_k)), \\
(\Psi_{N-k-1}^+)^{-1} &:= - \otimes_{\Lambda_k}^L \Tilt^+_{N-k-1} : \D(\modu(\Lambda_k)) \xrightarrow{\sim} \D(Y^+),
\end{align*}
and by compositing these equivalences, we have an equivalence
\[ \NKN_k := \RHom_Y(\Tilt_k, -) \otimes_{\Lambda_k}^L \Tilt^+_{N-k-1} : \D(Y) \xrightarrow{\sim} \D(Y^+). \]
By construction, the inverse of the equivalence $\NKN_k$ is given by
\[ (\NKN_k)^{-1} \simeq \NKN_{N-k-1}' := \RHom_{Y^+}(\Tilt^+_{N-k-1}, -) \otimes_{\Lambda_k}^L \Tilt_k. \]

On the other hand, the equivalence between $\D(Y)$ and $\D(Y^+)$ is first given by Kawamata and Namikawa in terms of the Fourier-Makai transform.
We recall their construction of Fourier-Makai type equivalences.
Let $\widetilde{Y}$ be a blowing-up of $Y$ at the zero section $\PP$.
Then, $\widetilde{Y}$ is also a blowing-up of $Y^+$ at $\PP^{\vee}$.
Since the normal bundle of $j : \PP \hookrightarrow Y$ is isomorphic to $\Omega_{\PP}^1$,
the exceptional divisor $E = \mathbb{P}_{\PP}(\Omega_{\PP}^1) \subset \widetilde{Y}$ can be embedded in the fiber product $\PP \times \PP^{\vee}$ by the Euler sequence.
We set $\hat{Y} := \widetilde{Y} \cup_E (\PP \times \PP^{\vee})$, and let $\hat{q} : \hat{Y} \to Y$ and $\hat{p} : \hat{Y} \to Y^+$ be projections.
\[ \begin{tikzcd}
 & \hat{Y} \arrow[rd, "\hat{p}"] \arrow[ld, "\hat{q}"'] & \\
 Y & & Y^+
 \end{tikzcd} \]
Let $\mathcal{L}_k$ be a line bundle on $\hat{Y}$ such that $\mathcal{L}_k|_{\widetilde{Y}} = \stsh_{\widetilde{Y}}(kE)$ and $\mathcal{L}_k|_{\PP \times \PP^{\vee}} = \stsh(-k, -k)$.
The Kawamata-Namikawa's functors are given by
\begin{align*}
\KN_k := R\hat{p}_*(L\hat{q}^*(-) \otimes \mathcal{L}_k) : \D(Y) \to \D(Y^+), \\
\KN'_k := R\hat{q}_*(L\hat{p}^*(-) \otimes \mathcal{L}_k) : \D(Y^+) \to \D(Y).
\end{align*}

The following result is due to Kawamata and Namikawa.

\begin{thm}[\cite{Ka02, Na03}] \label{KN}
The functors $\KN_k$ and $\KN'_k$ are equivalences.
\end{thm}

\begin{rem} \label{rem 3-1} \rm
By the definition of the functor $\KN_k$, the following diagram commutes
\[ \begin{tikzcd}
\D(Y) \arrow{r}{\KN_k} \arrow[d, "- \otimes \stsh_{Y}(1)"']& \D(Y^+) \arrow{d}{ - \otimes \stsh_{Y^+}(-1)} \\
\D(Y) \arrow[r, "\KN_{k+1}"'] & \D(Y^+).
\end{tikzcd} \]
The same holds for our equivalence $\NKN_k$ :
\[ \begin{tikzcd}
\D(Y) \arrow{r}{\NKN_k} \arrow[d, "- \otimes \stsh_{Y}(1)"']& \D(Y^+) \arrow{d}{ - \otimes \stsh_{Y^+}(-1)} \\
\D(Y) \arrow[r, "\NKN_{k+1}"'] & \D(Y^+).
\end{tikzcd} \]
\end{rem}

\begin{thm} \label{thm kn=nkn}
Our functor $\NKN_k$ (resp. $\NKN_k'$) coincides with the Kawamata-Namikawa's functor $\KN_k$ (resp. $\KN'_k$).
\end{thm}

Note that in the proof of Theorem \ref{thm kn=nkn}, we does not use the fact that the functors $\KN_k$ and $\KN_k'$ are equivalences.
Thus, our proof of Theorem \ref{thm kn=nkn} gives an alternative proof for Theorem \ref{KN} in this local model of the Mukai flop.

\begin{proof}
It is easy to see that $\KN'_{N-k-1}$ is the left and right adjoint of $\KN_k$.
Thus, it is enough to show that the following diagram commutes.
\[ \begin{tikzcd}
\D(Y) \arrow[rd, "\Psi_k"'] & & \D(Y^+) \arrow[ld, "\Psi_{N-k-1}^+"] \arrow[ll, "\KN'_{N-k-1}"'] \\
 & \D(\Lambda_k) &
\end{tikzcd} \]
We note that the composition $\Psi_k \circ \KN'_{N-k-1}$ is given by
\[ \RHom_{Y^+}(\KN_k(\Tilt_k), -) : \D(Y^+) \to \D(\Lambda_k). \]
Now, Theorem \ref{thm kn=nkn} follows from Lemma \ref{lem 4-3}.
\end{proof}

\begin{lem} \label{lem 4-3}
Let $k \in \mathbb{Z}$ a fixed integer. Then we have
\[ \KN_k(\stsh_{Y}(a)) \simeq \stsh_{Y^+}(-a) \]
for all $-N+k+1 \leq a \leq k$ and hence we have an isomorphism
\[ \KN_k(\Tilt_k) \simeq \Tilt^+_{N-k-1}. \]
\end{lem}

\begin{proof}
By Remark \ref{rem 3-1}, it is enough to show the isomorphism of functors for $k = 0$.
Recall that the correspondence $\widehat{Y}$ is given by $\widehat{Y} = \widetilde{Y} \cup_E \PP \times \PP^{\vee}$.
Hence, we have an exact sequence on $Y \times Y^+$
\[ 0 \to \stsh_{\widehat{Y}} \to \stsh_{\widetilde{Y}} \oplus \stsh_{\PP \times \PP^{\vee}} \to \stsh_E \to 0. \]
We use this sequence to compute the Fourier-Mukai functor $\KN_0 := \Phi_{\stsh_{\widehat{Y}}}$.
First, we have
\begin{align*}
\Phi_{\stsh_{\PP \times \PP^{\vee}}}(\stsh_{Y}(a)) &= R\Gamma(\PP, \stsh_{\PP}(a)) \otimes j'_*\stsh_{\PP^{\vee}} \\
&= \begin{cases}
j'_*\stsh_{\PP^{\vee}} &(\text{if $a = 0$}) \\
0 &(\text{if $-N + 1 \leq a < 0$}).
\end{cases} 
\end{align*}
The exceptional divisor $E \subset \widetilde{Y}$ is a universal hyperplane section over $P$ and hence a divisor on $P \times P^{\vee}$ of bi-degree $(1, 1)$.
Thus, we have an exact sequence
\[ 0 \to \stsh_{\PP \times \PP^{\vee}}(-1, -1) \to \stsh_{\PP \times \PP^{\vee}} \to \stsh_E \to 0. \]
From the same computation as above, we have
\begin{align*}
\Phi_{\stsh_{\PP \times \PP^{\vee}}(-1,-1)}(\stsh_{Y}(a)) &= R\Gamma(\PP, \stsh_{\PP}(a-1)) \otimes j'_*\stsh_{\PP^{\vee}}(-1) \\
&= \begin{cases}
0 &(\text{if $-N + 1 < a \leq 0$}) \\
j'_*\stsh_{\PP^{\vee}}(-1)[-N+1] &(\text{if $a = -N+1$}),
\end{cases} 
\end{align*}
and hence we have
\begin{align*}
\Phi_{\stsh_{E}}(\stsh_{Y}(a)) = \begin{cases}
j'_*\stsh_{\PP^{\vee}} &(\text{if $a = 0$}) \\
0 &(\text{if $-N + 1 < a < 0$}) \\
j'_*\stsh_{\PP^{\vee}}(-1)[-N+2] &(\text{if $a = -N+1$}).
\end{cases} 
\end{align*}
Furthermore, since we have
\[ \stsh_{\widetilde{Y}}(E) \simeq \tilde{q}^*\stsh_{Y}(-1) \otimes \tilde{p}^*\stsh_{Y^+}(-1), \]
and 
\begin{align*}
R\tilde{p}_*\stsh_E(kE) = \begin{cases}
%j'_*\stsh_{P^{\vee}} & \text{for $k=0$}, \\
0 & \text{for all $k = 1, \dots, N-2$}, \\
j'_*\stsh_{\PP^{\vee}}(-N)[-N+2] & \text{for $k= N - 1$},
\end{cases}
\end{align*}
we have
\begin{align*}
\Phi_{\stsh_{\widetilde{Y}}}(\stsh_Y(a)) &= R\tilde{p}_*(\stsh_{\widetilde{Y}}(-aE)) \otimes \stsh_{Y^+}(-a) \\
& = \stsh_{Y^+}(-a)
\end{align*}
for $-N+1 < a  \leq 0$, and $\Phi_{\stsh_{\widetilde{Y}}}(\stsh_Y(-N+1))$ lies on the exact triangle
\[ \stsh_{Y^+}(N-1) \to \Phi_{\stsh_{\widetilde{Y}}}(\stsh_Y(-N+1)) \to \stsh_{P^{\vee}}(-1)[-N+2]. \]
From the above, we can compute $\KN_0(\stsh_{Y}(a))$ for $-N+1 \leq a \leq 0$.
If $a = 0$, $\KN_0(\stsh_Y)$ lies on the exact triangle
\[ \KN_0(\stsh_Y) \to \stsh_{Y^+} \oplus j'_*\stsh_{\PP^{\vee}} \to  j'_*\stsh_{\PP^{\vee}}, \]
and hence we have
\[ \KN_0(\stsh_Y) \simeq \stsh_{Y^+}. \]
If $-N+1 < a < 0$, we have
\[ \KN_0(\stsh_Y(a)) \simeq \stsh_{Y^+}(-a). \]
Finally, if $a = -N+1$, $\KN_0(\stsh_Y(-N+1))$ lies on the exact triangle
\[ \KN_0(\stsh_Y(-N+1)) \to \Phi_{\stsh_{\widetilde{Y}}}(\stsh_Y(-N+1)) \to j'_*\stsh_{\PP^{\vee}}(-1)[-N+2]. \]
This triangle coincides with the above one that gives the object $\Phi_{\stsh_{\widetilde{Y}}}(\stsh_Y(-N+1))$ and hence we have
\[ \KN_0(\stsh_Y(-N+1)) \simeq \stsh_{Y^+}(N-1). \]
Thus, we have the isomorphism $\KN_0(\Tilt_0) \simeq \Tilt^+_N$ that we want.
\end{proof}

\subsection{P-twists and Mutations} \label{subsect mutation}
In this section, we introduce equivalences $\nu_{N+k}^-$ and $\nu_{N+k-1}^+$ between the derived categories of non-commutative algebras $\Lambda_{N+k}$ and $\Lambda_{N+k-1}$.
We show that a composition of multi-mutation functors $\nu_{N+k-1}^+ \circ \nu_{N+k}^-$ corresponds to an autoequivalence $P_k$ of $\D(Y)$ that is a \textit{P-twist} defined by a $\mathbb{P}^{N-1}$-object $j_*\stsh_{\PP}(k)$.

\subsubsection{Definition of multi-mutation}
First, we define a multi-mutation functor $\nu_{N-1}^- : \D(\modu(\Lambda_{N-1})) \to \D(\modu(\Lambda_{N-2}))$.
Recall that the algebras $\Lambda_{N-1}$ is given by
\[ \Lambda_{N-1} = \End_R \left(\bigoplus_{a=0}^{N-1} M_a \right). \]
Let us consider the canonical surjective morphism $R^{\oplus N} \twoheadrightarrow M_{-1}$. 
Note that this morphism is given by the push-forward of the canonical surjection $V \otimes_{\mathbb{C}} \stsh_{Y^+} \twoheadrightarrow \stsh_{Y^+}(1)$ by $\phi^+$.
Then, we define a $\Lambda_{N-1}$-module $C$ as
\[ C := \mathrm{Image}\left(\Hom_R(\bigoplus_{a=0}^{N-1} M_a, R^{\oplus N}) \to \Hom_R(\bigoplus_{a=0}^{N-1} M_a, M_{-1})\right), \]
and set a $\Lambda_{N-1}$-module $S$ as
\[ S := \Hom_{\Lambda_{N-1}}(\bigoplus_{a=0}^{N-1} M_a, \bigoplus_{a=0}^{N-2} M_a) \oplus C. \]

\begin{lem} \label{lem 3-1}
The following hold.
\begin{enumerate}
\item[(i)] There exists an isomorphism of $\Lambda_{N-1}$-modules
\[ S \simeq \RHom_{Y^+}(\Tilt_0^+, \Tilt_1^+). \]
\item[(ii)] The $\Lambda_{N-1}$-module $S$ defined above is a tilting generator of  the category $\D(\modu(\Lambda_{N-1}))$.
\item[(iii)] We have an isomorphism between $R$-algebras
\[ \End_{\Lambda_{N-1}}(S) \simeq \Lambda_{N-2}. \]
\end{enumerate}
\end{lem}

\begin{proof}
The (ii) and (iii) follow from (i).
First, we have
\[ \RHom_{Y^+}(\Tilt_0^+, \Tilt_1^+) = \RHom_{Y^+}(\Tilt_0^+, \bigoplus_{a=-N+2}^0\stsh_{Y^+}(a)) \oplus \RHom_{Y^+}(\Tilt_0^+, \stsh_{Y^+}(1)). \]
As explained above, we have
\[ M_{-a} = \phi^+_* \stsh_{Y^+}(a) \]
for all $-N+1 \leq a \leq N-1$, and we have
\begin{align*}
\RHom_{Y^+}(\Tilt_0^+, \bigoplus_{a=-N+2}^0 \stsh_{Y^+}(a)) &= \Hom_{Y^+}(\Tilt_0^+, \bigoplus_{a=-N+2}^0\stsh_{Y^+}(a)) \\
&= \Hom_{R}(\bigoplus_{a=0}^{N-1} M_a, \bigoplus_{a=0}^{N-2} M_a).
\end{align*}
Next, since the sheaf $\lhom_{Y^+}(\Tilt_0^+, \stsh_{Y^+}(1))$ on $Y^+$ is a vector bundle and hence is torsion free,
the $R$-module $\Hom_{Y^+}(\Tilt_0^+, \stsh_{Y^+}(1)) = \phi^+_*\lhom_{Y^+}(\Tilt_0^+, \stsh_{Y^+}(1))$ is also torsion free.
Since two $R$-modules $\Hom_R(\bigoplus_{a=0}^{N-1} M_a, M_{-1})$ and $\Hom_{Y^+}(\Tilt_0^+, \stsh_{Y^+}(1))$ are isomorphic in codimension one,
the natural map
\[ \Hom_{Y^+}(\Tilt_0^+, \stsh_{Y^+}(1)) \to \Hom_R(\bigoplus_{a=0}^{N-1} M_a, M_{-1}) \]
is injective.
Let us consider the surjective morphism $V \otimes_{\mathbb{C}} \stsh_{Y^+} \to \stsh_{Y^+}(1)$.
We note that the map
\[ \Hom_{Y^+}(\Tilt_0^+, V \otimes_{\mathbb{C}} \stsh_{Y^+}) \to \Hom_{Y^+}(\Tilt_0^+, \stsh_{Y^+}(1)) \]
is surjective because we have a vanishing of an extension
\[ \Ext^1_{Y^+}(\Tilt_0^+, \pi'^*\Omega_{\PP^{\vee}}(1)) = H^1(Y^+, \bigoplus_{a=0}^{N-1}\pi'^*\Omega_{\PP^{\vee}}(a+1)) = 0 \]
from the same argument as in the proof of Corollary \ref{cor alt NCCR MCM}.
Thus, we have the following commutative diagram
\[ \begin{tikzcd}
\Hom_{Y^+}(\Tilt_0^+, V \otimes_{\mathbb{C}} \stsh_{Y^+}) \arrow{r}{\simeq} \arrow[d, twoheadrightarrow] & \Hom_{R}(\bigoplus_{a=0}^{N-1} M_a, V \otimes_{\mathbb{C}} M_0) \arrow{d} \\
\Hom_{Y^+}(\Tilt_0^+, \stsh_{Y^+}(1)) \arrow[r, hook] & \Hom_{R}(\bigoplus_{a=0}^{N-1} M_a, M_{-1})
\end{tikzcd} \]
and hence we have $\RHom_{Y^+}(\Tilt_0^+, \stsh_{Y^+}(1)) = \Hom_{Y^+}(\Tilt_0^+, \stsh_{Y^+}(1)) = C$.
\end{proof}

From the above lemma, we can define the equivalence:
\begin{defi} \rm \label{defi mutation} We set
\[ \nu_{N-1}^- := \RHom_{\Lambda_{N-1}}(S, -) : \D(\modu(\Lambda_{N-1})) \xrightarrow{\sim} \D(\modu(\Lambda_{N-2})). \]
We call this functor $\nu_{N-1}^-$ the \textit{multi-mutation functor}.
By Lemma \ref{lem 3-1}, multi-mutation $\nu_{N-1}^-$ coincides with the functor
\[ \RHom_{\Lambda_{N-1}}(\RHom_{Y^+}(\Tilt_{0}^+, \Tilt_{1}^+), -) : \D(\modu(\Lambda_{N-1})) \xrightarrow{\sim} \D(\modu(\Lambda_{N-2})), \]
and hence the following diagram commutes
\[ \begin{tikzcd}
\D(Y^+) \arrow{r}{\Psi^+_{N-1}} \arrow[rd, "\Psi^+_{N-2}"'] & \D(\modu(\Lambda_{k})) \arrow{d}{\nu_{N-1}^-} \\
 & \D(\modu(\Lambda_{N-2})).
\end{tikzcd} \]
We also define a multi-mutation functor $\nu_k^- :  \D(\modu(\Lambda_{k})) \xrightarrow{\sim} \D(\modu(\Lambda_{k-1}))$ by using the following commutative diagram.
\[ \begin{tikzcd}
\D(Y^+) \arrow[r, "\otimes \stsh(-N+k+1)"'] \arrow[d, "\Psi^+_{N-k-1}"] \arrow[rrr, bend left=15, "\id"'] & \D(Y^+) \arrow[r, equal] \arrow[d, "\Psi^+_{0}"] & \D(Y^+) \arrow[r, "\otimes \stsh(N-k-1)"'] \arrow[d, "\Psi^+_{1}"] & \D(Y^+) \arrow[d, "\Psi^+_{N-k}"] \\
\D(\modu(\Lambda_k)) \arrow[r, "F_{N-1}^k"] \arrow[rrr, bend right=15, "\nu_k^-"] & \D(\modu(\Lambda_{N-1})) \arrow[r, "\nu_{N-1}^-"] & \D(\modu(\Lambda_{N-2})) \arrow[r, "F_{k-1}^{N-2}"] & \D(\modu(\Lambda_{k-1})),
\end{tikzcd} \]
where the functor $F^i_j : \D(\modu(\Lambda_{i})) \to \D(\modu(\Lambda_{j}))$ is given by the composition
\small{\[ F^i_j : \D(\modu(\Lambda_{i})) \xrightarrow{-\otimes_{\Lambda_i} \Tilt^+_{N-i-1}} \D(Y^+) \xrightarrow{-\otimes \stsh_Y(i-j)} \D(Y^+) \xrightarrow{\RHom_Y(\Tilt^+_{N-j-1},-)} \D(\modu(\Lambda_{j})). \]}
\normalsize
%We also call the functor $\nu_k^-$ the \textit{multi-mutation}.
\end{defi}

Applying the same argument for the side of $Y$,
we can define a multi-mutation functor $\nu_{k}^+ : \D(\modu(\Lambda_{k})) \to \D(\modu(\Lambda_{k+1}))$.
Again, by construction, we can show that there is a commutative diagram of functors
\[ \begin{tikzcd}
\D(Y) \arrow{r}{\Psi_{k}} \arrow[rd, "\Psi_{k+1}"'] & \D(\modu(\Lambda_{k})) \arrow{d}{\nu_{k}^+} \\
 & \D(\modu(\Lambda_{k+1})).
\end{tikzcd} \]

\subsubsection{Connection between multi-mutations and IW mutations} \label{subsubsect IW-multi}
In the following, 
we explain the multi-mutation functor $\nu_{N-1}^-$ is given by a composition of IW mutations.
For definitions and basic properties of IW mutations, see Section \ref{subsect IW}.
Let us consider the long Euler sequence on $\PP^{\vee} = \mathbb{P}(V^*)$
\small{
\begin{align*}
0 \to \stsh_{\PP^{\vee}}(-N+1) \to V^* \otimes_{\mathbb{C}} \stsh_{\PP^{\vee}}(-N+2) \to \bigwedge^{N-2}V \otimes_{\mathbb{C}} \stsh_{\PP^{\vee}}(-N+3) \to \cdots \\
 \to \bigwedge^2 V \otimes_{\mathbb{C}} \stsh_{\PP^{\vee}}(-1) \to V \otimes_{\mathbb{C}} \stsh_{\PP^{\vee}} \to \stsh_{\PP^{\vee}}(1) \to 0. 
\end{align*}}
\normalsize By applying the functor $(\phi^+)_* \circ (\pi')^*$ to the above sequence, we have a resolution of the module $M_{-1}$ by other modules $M_0, \dots, M_{N-1}$:
\small{ \[ 0 \to M_{N-1} \to V^* \otimes_{\mathbb{C}} M_{N-2} \to \bigwedge^{N-2}V \otimes_{\mathbb{C}} M_{N-3} \to \cdots \to \bigwedge^2 V \otimes_{\mathbb{C}} M_1 \to V \otimes_{\mathbb{C}} M_0 \to M_{-1} \to 0. \]}
\normalsize We splice this sequence into short exact sequences
\begin{align*}
0 \to M_{N-1} \to V^* &\otimes_{\mathbb{C}} M_{N-2} \to L_{N-2} \to 0 \\
0 \to L_{N-2} \to \bigwedge^{N-2} V &\otimes_{\mathbb{C}} M_{N-3} \to L_{N-3} \to 0 \\
\vdots \\
0 \to L_k \to \bigwedge^k V &\otimes_{\mathbb{C}} M_{k-1} \to L_{k-1} \to 0 \\
&\vdots \\
0 \to L_1 \to V &\otimes_{\mathbb{C}} M_0 \to M_{-1} \to 0 
\end{align*}
and set $L_{N-1} := M_{N-1}$, $L_0 := M_{-1}$, $W := \bigoplus_{a = 0}^{N-2} M_a$, and $E_k := W \oplus L_k$.
By dualizing above morphisms, we have a map $\bigwedge^k V^* \otimes_{\mathbb{C}} M_{k-1}^* \to L_k^*$.
Since the module $M_a$ is reflexive, the above map is surjective.
Then, applying the functor $- \oplus W^*$, we have a surjective map
\[ \left( \bigwedge^k V^* \otimes_{\mathbb{C}} M_{k-1}^* \right) \oplus W^* \twoheadrightarrow E_k^*. \]
First, we prove the following

\begin{lem} \label{lem approximation}
The map $\left(\bigwedge^k V^* \otimes_{\mathbb{C}} M_{k-1}^* \right) \oplus W^* \to E_k^*$ is a right $(\add W^*)$-approximation.
\end{lem}

\begin{proof}
Let us consider the exact sequence
\[ 0 \to L^*_{k-1} \to \bigwedge^k V^* \otimes_{\mathbb{C}} M^*_{k-1} \to L_k^* \to 0. \]
We have to show that the map
\[ \Hom_R(W^*, \bigwedge^k V^* \otimes_{\mathbb{C}} M^*_{k-1}) \to \Hom_R(W^*, L_k^*) \]
is surjective.
First, by definition, we have $M^*_{k-1} \simeq M_{-k+1} \simeq (\phi^+)_*\stsh_{Y^+}(k-1)$.
On $Y^+$, there is a canonical short exact sequence
\[ 0 \to (\pi')^*\bigwedge^{k-1}T_{\PP^{\vee}} \otimes \stsh_{Y^+}(-1) \to \bigwedge^k V^* \otimes_{\mathbb{C}} \stsh_{Y^+}(k-1) \to (\pi')^*\bigwedge^{k} T_{\PP^{\vee}} \otimes \stsh_{Y^+}(-1) \to 0, \]
where $T_{\PP^{\vee}}$  is the tangent bundle on $Y^+$.
Put $\mathcal{H}_k^* :=  (\pi')^*\bigwedge^{k} T_{\PP^{\vee}} \otimes \stsh_{Y^+}(-1)$ and $\mathcal{H}_k := \mathcal{H}_k^{**}$.
Since the first non-trivial term of the above exact sequence does not have higher cohomology, we have an isomorphism
\[ L_k^* \simeq (\phi^+)_* \left( \mathcal{H}_k^* \right) \]
by induction on $k$.
Furthermore, since the third non-trivial term of the above exact sequence and its dual have no higher cohomology,
it follows from Lemma \ref{lem 3-A} that the module $L_k^*$ is (maximal) Cohen-Macaulay,
and hence, the module $\Hom_R(W^*, L_k^*)$ is reflexive by Proposition \ref{CM ref 2}.
In addition, by Proposition \ref{CM ref 1}, Lemma \ref{lem 3-A}, and Proposition \ref{tilting 3}, the module
\[ \Hom_{Y^+}\left(\bigoplus_{a = 0}^{N-2} \stsh_{Y^+}(a),\mathcal{H}_k^* \right) \]
is also reflexive.
Therefore, we have an isomorphism
\[ \Hom_{Y^+}\left(\bigoplus_{a = 0}^{N-2} \stsh_{Y^+}(a), \mathcal{H}_k^* \right) \simeq \Hom_R(W^*, L_k^*). \]
On the other hand, again by Proposition \ref{tilting 3}, we have the vanishing of an extension group
\[ \Ext^1_{Y^+}\left(\bigoplus_{a = 0}^{N-2} \stsh_{Y^+}(a), \mathcal{H}_{k-1}^* \right) = 0. \]
This vanishing says that the map
\small{\[ \Hom_{Y^+}\left(\bigoplus_{a = 0}^{N-2} \stsh_{Y^+}(a), \bigwedge^{k} V^* \otimes_{\mathbb{C}} \stsh_{Y^+}(k-1) \right) \to
\Hom_{Y^+}\left(\bigoplus_{a = 0}^{N-2} \stsh_{Y^+}(a), \mathcal{H}_k^* \right) \]}
\normalsize
is surjective.
Thus, we have the morphism
\[ \Hom_R(W^*, \bigwedge^k V^* \otimes_{\mathbb{C}} M_{k-1}^*) \to \Hom_R(W^*, L_k^*) \]
is also surjective.
\end{proof}

Since the kernel of the approximation $\left(\bigwedge^k V^* \otimes_{\mathbb{C}} M_{k-1}^* \right) \oplus W^* \to E_k^*$ is isomorphic to $L_{k-1}^*$,
the $R$-module $E_{k-1}$ is isomorphic to a (left) IW mutation $\mu_W^L(E_k)$ of $E_k$ at $W$.
Thus, by Theorem \ref{IW mutation equiv}, we have a derived equivalence
\[ \IW_W : \D(\modu(\End_{R}(E_k))) \xrightarrow{\sim} \D(\modu(\End_R(E_{k-1}))). \]
However, in this case, we can show directly that the functor $\IW_W$ actually gives an equivalence of categories.
As in the proof of Lemma \ref{lem approximation}, put $\mathcal{H}_k := (\pi')^*\Omega_{\PP^{\vee}}^k \otimes \stsh_{Y^+}(1)$.

\begin{lem} \label{lem IWmu}
\begin{enumerate}
\item[(1)] We have an isomorphism of $R$-algebras $\End_R(E_k) \simeq \End_{Y^+}(\mathcal{S}^+_k)$,
where $\mathcal{S}^+_k := \bigoplus_{-N+2}^0 \stsh_{Y^+}(a) \oplus \mathcal{H}_k$ is a tilting bundle on $Y^+$ that is given in Proposition \ref{tilting 3}.
\item[(2)] We have an isomorphism of functors
\[ \IW_W \simeq \RHom_{\End_R(E_k)}(\RHom_{Y^+}(\mathcal{S}_k^+, \mathcal{S}_{k-1}^+), -). \]
\item[(3)] In particular, IW mutation functor $\IW_W$ gives an equivalence of categories,
and the following diagram of functors commutes
\[ \begin{tikzcd}
\D(Y^+) \arrow[r, "S_k"] \arrow[rd, "S_{k-1}"'] & \D(\modu(\End_R(E_k))) \arrow{d}{\IW_W} \\
 & \D(\modu(\End_R(E_{k-1}))),
\end{tikzcd} \]
where $S_k :=  \RHom_{Y^+}(\mathcal{S}^+_k, -) : \D(Y^+) \to \D(\modu(\End_R(E_k)))$.
\end{enumerate}
\end{lem}

\begin{proof}
We can prove this lemma by using almost same arguments as in Lemma \ref{lem 3-1}.
The different point from Lemma \ref{lem 3-1} is that
the vanishing of $\Ext_{Y^+}^i(\mathcal{S}^+_k, \mathcal{S}^+_{k-1})$ for $i > 0$ is non-trivial.
However, this vanishing follows from direct computations using Proposition \ref{BLV10 lem}.
\end{proof}

Now we ready to prove the following result that gives a correspondence between multi-mutations and IW mutations.

\begin{thm} \label{thank Wemyss}
An equivalence obtained by composing $N-1$ IW mutation functors
\[ \IW_W \circ \IW_W \circ \cdots \circ \IW_W : \D(\modu(\Lambda_{N-1})) \to \D(\modu(\Lambda_{N-2})) \]
is isomorphic to a multi-mutation functor $\nu_{N-1}^-$.
\end{thm}

Here, we note that $\End_{R}(E_{N-1}) = \Lambda_{N-1}$ and $\End_{R}(E_0) = \Lambda_{N-2}$.

\begin{proof}
By Lemma \ref{lem 3-1} (3), we have a commutative diagram
\small{\[ \begin{tikzcd}
\D(Y^+) \arrow[d, "S_{N-1}"] \arrow[rd, "S_{N-2}"] \arrow[rrrd, bend left=10, "S_0"] & & &  \\
\D(\modu(\End_R(E_{N-1}))) \arrow[r, "\IW_W"'] & \D(\modu(\End_R(E_{N-2}))) \arrow[r, "\IW_W"'] & \cdots \arrow[r, "\IW_W"'] & \D(\modu(\End_R(E_{0}))),
\end{tikzcd} \]}
\normalsize Hence, we have $\IW_W \circ \IW_W \circ \cdots \circ \IW_W \simeq S_0 \circ S_{N-1}^{-1} = \Psi^+_{1} \circ (\Psi^+_{0})^{-1} \simeq \nu^-_{N-1}$.
\end{proof}

\begin{rem} \label{rem IW=multi} \rm
Applying the same argument, we can prove that a multi-mutation functor
\[ \nu_k^- : \D(\modu(\Lambda_k)) \to \D(\modu(\Lambda_{k-1})) \]
is written as a composition of IW mutation functors if
$1 \leq k \leq N-1$.
In other cases, the above argument cannot be applied because we only know that the module $\bigoplus_{a = -N+k+1}^k M_a$ gives an NCCR if $0 \leq k \leq N-1$ (see Theorem \ref{NCCR1}).
\end{rem}

Next, we discuss the case of multi-mutations $\nu_k^+$.

\begin{thm}
A multi-mutation functor
\[ \nu_{N-2}^+ : \D(\modu(\Lambda_{N-2})) \to \D(\modu(\Lambda_{N-1})) \]
can be written as a composition of $N-1$ IW mutation functors.
\end{thm}

\begin{proof}
Let us consider the long Euler sequence on $\PP$
\small{
\begin{align*}
&0 \to \stsh_{\PP}(-1) \to V \otimes_{\mathbb{C}} \stsh_{\PP} \to \bigwedge^{N-2}V^* \otimes_{\mathbb{C}} \stsh_{\PP}(1) \to \cdots \\
&\to \bigwedge^2 V^* \otimes_{\mathbb{C}} \stsh_{\PP}(N-3) \to V^* \otimes_{\mathbb{C}} \stsh_{\PP}(N-2) \to \stsh_{\PP}(N-1) \to 0. 
\end{align*}}
\normalsize Applying a functor $\phi_* \circ \pi^*$, we have a long exact sequence
\small{ \begin{align*} 0 \to M_{-1} \to V \otimes_{\mathbb{C}} R &\to \bigwedge^{N-2}V^* \otimes_{\mathbb{C}} M_{1} \to \cdots  \\
&\to \bigwedge^2 V^* \otimes_{\mathbb{C}} M_{N-3} \to V^* \otimes_{\mathbb{C}} M_{N-2} \to M_{N-1} \to 0. \end{align*}}
\normalsize
Using completely same argument as in the proof of Theorem \ref{thank Wemyss},
we have an equivalence of categories
\[ \IW_{W} \circ \IW_{W} \circ \cdots \circ \IW_{W} : \D(\modu(\Lambda_{N-2})) \to \D(\modu(\Lambda_{N-1})) \]
and this functor is isomorphic to the functor $\Psi_{N-1} \circ \Psi_{N-2}^{-1} \simeq \nu_{N-2}^+$ under the above identification of algebras.
\end{proof}

\begin{rem} \rm
As in Remark \ref{rem IW=multi}, we can show that a multi-mutation functor $\nu_k^+ : \D(\modu(\Lambda_k)) \to \D(\modu(\Lambda_{k+1}))$ can be described as a composition of IW mutation functors if $0 \leq k \leq N-2$.
\end{rem}

\begin{rem} \rm
From the proof of theorems, we notice that the object
\[ \mu_W^L(\mu^L_W( \cdots (\mu_W^L(\bigoplus_{a=0}^{N-1} M_a)) \cdots )), \]
which obtained from $\bigoplus_{a=0}^{N-1} M_a$ after taking IW mutations at $W$ $(2N-2)$-times, coincides with the original module $\bigoplus_{a=0}^{N-1} M_a$:
\[ \mu_W^L(\mu^L_W( \cdots (\mu_W^L(\bigoplus_{a=0}^{N-1} M_a)) \cdots )) = \bigoplus_{a=0}^{N-1} M_a. \]

If the ring $R$ is complete normal $3$-sCY and $M$ is a maximal modifying module\footnote{For the definition of maximal modifying $R$-modules, see \cite[Definition 4.1]{IW14}. We note that a module that gives an NCCR is a maximal modifying module if $R$ is a normal $d$-sCY ring \cite[Proposition 4.5]{IW14}.},
Iyama and Wemyss proved that two times mutation $\mu^L_N\mu^L_N(M)$ of $M$ at an indecomposable summand $N$ coincides with $M$ \cite[Summary 6.25]{IW14}:
\[ \mu^L_N\mu^L_N(M) = M. \]
Although the module $W$ that we used for mutations is not indecomposable, I think we can regard our equality of modules as a generalization of Iyama-Wemyss's one.
The number of mutations we need seems to be related to the dimension of a fiber of a crepant resolution (or $\mathbb{Q}$-factorial terminalization).
\end{rem}

\begin{cor}
The equivalence from $\D(Y)$ to $\D(Y^+)$ obtained by the composition
\[ \D(Y) \xrightarrow{\Psi_0} \D(\modu \Lambda_0) \xrightarrow{\nu^+_{N-2} \circ \cdots \circ \nu^+_0} \D(\modu \Lambda_{N-1}) \xrightarrow{(\Psi^+_0)^{-1}} \D(Y^+) \]
is the inverse of the (original) Kawamata-Namikawa's functor $\KN'_0$.
\end{cor}

By the above remark, the functor $\nu^+_{N-2} \circ \cdots \circ \nu^+_0$ can be written as the composition of $(N-1)^{N-1}$ IW mutation functors.
On the other hand, two tilting bundles $\Tilt_0$ and $\Tilt_0^+$ provide projective generators of the perverse hearts ${}^0\mathrm{Per}(Y/A_{N-2})$ and ${}^0\mathrm{Per}(Y^+/A^o_{N-2})$ respectively (see \cite[Example 5.3]{TU10}).
Please compare this corollary with \cite[Theorem 4.2]{We14}.

\subsubsection{Multi-mutations and P-twists}
Next, we explain that a composition of two multi-mutation functors corresponds to a P-twist on $\D(Y)$.
First, we recall that the object $j_*\stsh_{\PP}(k)$ is a $\mathbb{P}^{N-1}$-object in $\D(Y)$.
This fact is well-known but I give the proof here for reader's convenience.

\begin{lem} \label{P-obj lem}
$j_*\stsh_{\PP}(k)$ is a $\mathbb{P}^{N-1}$-object in $\D(Y)$.
\end{lem}

\begin{proof}
It is enough to show the case if $k=0$.
Let us consider the spectral sequence
\[ E^{p,q}_2 = H^p(Y, \lext_Y^q(j_*\stsh_{\PP}, j_*\stsh_{\PP})) \Rightarrow \Ext_Y^{p+q}(j_*\stsh_{\PP}, j_*\stsh_{\PP}). \]
Since there is an isomorphism
\[ \lext_Y^q(j_*\stsh_{\PP}, j_*\stsh_{\PP}) \simeq j_*\bigwedge^q \mathcal{N}_{\PP/Y} \simeq j_*\Omega^q_{\PP}, \]
we have
\begin{align*}
E_2^{p, q} &= H^p(\PP, \Omega^q_{\PP}) \\
&= \begin{cases}
\mathbb{C} & \text{if $0 \leq p = q \leq N -1$,} \\
0 & \text{otherwise.}
\end{cases}
\end{align*}
Therefore, we have
\begin{align*}
\Ext_Y^{i}(j_*\stsh_{\PP}, j_*\stsh_{\PP}) = \begin{cases}
\mathbb{C} & \text{if $i = 2k$ and $k = 0 , \dots, N-1$}, \\
0 & \text{otherwise}.
\end{cases}
\end{align*}
This shows the lemma.
\end{proof}

\begin{defi} \rm
Let $P_k$ be a P-twist that is defied by the $\mathbb{P}^{N-1}$-object $j_*\stsh_P(k)$.
More explicitly, the functor $P_k$ is given by
\[ P_k(E) = \Cone\left( \Cone\left(j_*\stsh_{\PP}(k)[-2] \to j_*\stsh_{\PP}(k)\right) \otimes_{\mathbb{C}} \RHom_Y(j_*\stsh_{\PP}(k), E) \xrightarrow{\mathrm{ev}} E \right). \]
\end{defi}

\begin{rem} \rm
By the definition of the functor $P_k$, the following diagram commutes.
\[ \begin{tikzcd}
\D(Y) \arrow[r, "P_k"] \arrow[d, "- \otimes \stsh_{Y}(-1)"'] & \D(Y) \arrow[d, "- \otimes \stsh_{Y}(-1)"] \\
\D(Y) \arrow[r, "P_{k-1}"] & \D(Y)
\end{tikzcd} \]
\end{rem}

The following  is one of main results in this paper.

\begin{thm} \label{thm mutation-Ptwist}
The following diagram of equivalence functors commutes
\[ \begin{tikzcd}
\D(Y) \arrow{d}{P_k} \arrow{r}{\Psi_{N+k}} & \D(\modu(\Lambda_{N+k})) \arrow{d}{\nu_{N+k}^-}\\
\D(Y) \arrow[r, "\Psi_{N+k-1}"] \arrow[d, equal] & \D(\modu(\Lambda_{N+k-1})) \arrow[d, "\nu_{N+k-1}^+"] \\
\D(Y) \arrow[r, "\Psi_{N+k}"] & \D(\modu(\Lambda_{N+k})).
\end{tikzcd} \]
In particular, if we fix the identification $\Psi_{N+k} : \D(Y) \to \D(\modu(\Lambda_{N+k}))$, 
a composition of two multi-mutation functors
\[ \nu_{N+k-1}^+ \circ \nu_{N+k}^- \in \Auteq(\D(\modu(\Lambda_{N+k}))) \]
corresponds to a P-twist $P_k \in \Auteq(\D(Y))$.
\end{thm}

\begin{rem} \rm
If $1 \leq N + k \leq N-1$ (i.e. if $- N + 1 \leq k \leq -1$), 
multi-mutation functors $\nu_{N+k}^-$ and $\nu_{N+k-1}^+$ are can be written as compositions of IW mutation functors.
Thus, in the case of Mukai flops, we can interpret a P-twist on $Y$ as a composition of many IW mutation functors.
This is a higher dimensional generalization of the result of Donovan and Wemyss \cite{DW16}.
\end{rem}

\begin{proof}[Proof of Theorem \ref{thm mutation-Ptwist}]
It is enough to show the theorem for one $k$.
Here, we prove the case if $k = -1$.
Recall that the composition
\[ \D(Y) \xrightarrow{\Psi_{N-1}} \D(\modu(\Lambda_{N-1})) \xrightarrow{\nu^-_{N-1}} \D(\modu(\Lambda_{N-2})) \]
coincides with the functor
\[ \RHom_Y(S \otimes_{\Lambda_{N-1}} \Tilt_{N-1},-) : \D(Y) \to \D(\modu(\Lambda_{N-2})). \]
By Theorem \ref{thm kn=nkn} and Lemma \ref{lem 3-1}, we have $S \otimes_{\Lambda_{N-1}} \Tilt_{N-1} \simeq \KN'_0(\Tilt^+_1)$.
On the other hand, the equivalence that is given by the composition of functors
\[ \D(Y) \xrightarrow{P_{-1}} \D(Y) \xrightarrow{\Psi_{N-2}} \D(\modu(\Lambda_{N-2})) \]
coincides with the functor that is given by
\[ \RHom_Y\left((P_{-1})^{-1}(\Tilt_{N-2}),-\right) : \D(Y) \to \D(\modu(\Lambda_{N-2})). \]
Thus, we have to show that 
\[ P_{-1}(\KN'_0(\Tilt^+_1)) \simeq \Tilt_{N-2}. \]
Recall that the tilting bundles are given by
\[ \Tilt_{N-2} = \bigoplus_{a=-1}^{N-2} \stsh_Y(a), ~~~ \Tilt^+_{1} = \bigoplus_{a= - N + 2}^1 \stsh_{Y^+}(a). \]
By Lemma \ref{lem 4-3}, we have
\[ \KN'_0(\stsh_{Y^+}(a)) \simeq \stsh_Y(-a) \]
for $-N+2 \leq a \leq 0$.
Therefore, we have to compute the object $\KN'_0(\stsh_{Y^+}(1))$.
As in Lemma \ref{lem 4-3}, we use the exact sequence
\[ 0 \to \stsh_{\hat{Y}} \to \stsh_{\widetilde{Y}} \oplus \stsh_{\PP \times \PP^{\vee}} \to \stsh_E \to 0. \]
An easy computation shows that we have
\begin{align*}
\Phi_{\stsh_{\PP \times \PP^{\vee}}}^{Y^+ \to Y}(\stsh_{Y^+}(1)) &\simeq V \otimes_{\mathbb{C}} j_*\stsh_{\PP}, \\
\Phi_{\stsh_E}^{Y^+ \to Y}(\stsh_{Y^+}(1)) &\simeq  j_*T_{\PP}(-1), \\
\Phi_{\stsh_{\widetilde{Y}}}^{Y^+ \to Y}(\stsh_{Y^+}(1)) &\simeq I_{\PP/Y}(-1),
\end{align*}
where $T_{\PP}$ is the tangent bundle on $\PP = \mathbb{P}(V)$ and $I_{
\PP/Y}$ the ideal sheaf of $j : \PP \subset Y$.
Thus, we have the following exact triangle
\[ \KN'_0(\stsh_{Y^+}(1)) \to I_{\PP/Y}(-1) \oplus (V \otimes_{\mathbb{C}} j_*\stsh_{\PP}) \to   j_*T_{\PP}(-1). \]
By combining this triangle with the split triangle
\[ V \otimes_{\mathbb{C}} j_*\stsh_{\PP} \to I_{\PP/Y}(-1) \oplus (V \otimes_{\mathbb{C}} j_*\stsh_{\PP}) \to I_{\PP/Y}(-1), \]
we have the following diagram
\[ \begin{tikzcd}
j_*\stsh_{\PP}(-1) \arrow{d} \arrow{r} & V \otimes_{\mathbb{C}} j_*\stsh_{\PP} \arrow{d} \arrow{r} & j_*T_{\PP}(-1) \arrow[d, equal] \\
\KN'_0(\stsh_{Y^+}(1)) \arrow{d} \arrow{r} & I_{\PP/Y}(-1) \oplus (V \otimes_{\mathbb{C}} j_*\stsh_{\PP}) \arrow{d} \arrow{r} & j_*T_{\PP}(-1) \\
I_{\PP/Y}(-1) \arrow[r, equal] & I_{\PP/Y}(-1) & 
\end{tikzcd} \]
Hence, the object $\KN'_0(\stsh_{Y^+}(1)) \in \D(Y)$ is a sheaf,
and if we set $\mathcal{F} = \KN'_0(\stsh_{Y^+}(1))$,
the sheaf $\mathcal{F}$ lies on the exact sequence
\[ 0 \to j_*\stsh_{\PP}(-1) \to \mathcal{F} \to \stsh_Y(-1) \to j_*\stsh_{\PP}(-1) \to 0. \]
Recall that $j_*\stsh_{\PP}(-1)$ is a $\mathbb{P}^{N-1}$ object that defines the P-twist $P_{-1}$.
In particular, $\Ext^2_Y(j_*\stsh_{\PP}(-1), j_*\stsh_{\PP}(-1)) = \mathbb{C} \cdot h$.
Let $C(h)$ be an object in $\D(Y)$ that lies on the exact triangle
\[ j_*\stsh_{\PP}(-1)[-2] \xrightarrow{h} j_*\stsh_{\PP}(-1) \to C(h). \]
Then, we have an exact triangle
\[ \stsh_Y(-1)[-1] \to C(h)  \to \mathcal{F}. \]
Let $e : C(h) \to \mathcal{F}$ be the morphism that appears in the above triangle.

Next, we compute the objects $P_{-1}(\KN'_0(\stsh_{Y^+}(a)))$ for $-N+2 \leq a \leq 1$.
Recall that the P-twist $P_{-1}$ is given by
\[ P_{-1}(E) := \Cone(C(h) \otimes_{\mathbb{C}} \RHom_Y(j_*\stsh_{\PP}(-1), E) \to E). \]
Since we have
\begin{align*}
\RHom_Y(j_*\stsh_{\PP}(-1), \stsh_Y(b)) &\simeq \RHom_{\PP}(\stsh_{\PP}(-1), j^!\stsh_Y(b)) \\
&\simeq R\Gamma(\PP, \stsh_{\PP}(-N+b+1))[-N+1]
\end{align*}
by adjunction, we have
\[ \RHom_Y(j_*\stsh_{\PP}(-1), \stsh_Y(b)) = 0 \]
for $0 \leq b \leq N-2$, and hence we have
\[ P_{-1}(\KN'_0(\stsh_{Y^+}(a))) \simeq P_{-1}(\stsh_Y(-a)) = \stsh_Y(-a) \]
for $-N+2 \leq a \leq 0$.
It is remaining to compute the object $P_{-1}(\KN'_0(\stsh_{Y^+}(1))) = P_{-1}(\mathcal{F})$.
From the above computation, we have
\[ \RHom_Y(j_*\stsh_{\PP}(-1), \stsh_Y(-1)) \simeq R\Gamma(\PP, \stsh_{\PP}(-N))[-N+1] \simeq \mathbb{C}[-2N+2]. \]
On the other hand, by the exact triangle
\[ j_*\stsh_{\PP}(-1)[-2] \xrightarrow{h} j_*\stsh_{\PP}(-1) \to C(h) \]
that defines $C(h)$ and the computation
\[ \RHom_Y(j_*\stsh_{\PP}(-1), j_*\stsh_{\PP}(-1)) = \bigoplus_{i=0}^{N-1} \mathbb{C}[-2i], \]
we have
\[ \RHom_Y(j_*\stsh_{\PP}(-1), C(h)) = \mathbb{C} \oplus \mathbb{C}[-2N+1]. \]
Hence, by the exact triangle
\[ \stsh_Y(-1)[-1] \to C(h)  \xrightarrow{e} \mathcal{F} \]
that we obtained above, we have
\[ \RHom_Y(j_*\stsh_{\PP}(-1), \mathcal{F}) = \mathbb{C}, \]
and thus, the object $P_{-1}(\mathcal{F})$ lies on the exact triangle
\[ C(h) \xrightarrow{\mathrm{ev}} \mathcal{F} \to P_{-1}(\mathcal{F}). \]
Since we have
\[ \Hom_Y(C(h), \mathcal{F}) \simeq \mathbb{C} \]
from the above, we have the map $\mathrm{ev} : C(h) \to \mathcal{F}$ coincides with the map $e : C(h) \to \mathcal{F}$ up to non-zero scaler.
Therefore, we have
\[ P_{-1}(\mathcal{F}) \simeq \stsh_Y(-1) \]
and hence $P_{-1}(\KN'_0(\Tilt^+_1)) \simeq \Tilt_{N-2}$. This is what we want.
\end{proof}

Theorem \ref{thm mutation-Ptwist} recovers the following result that was first proved by Cautis, and later Addington-Donovan-Meachan in different ways.
Our approach that uses non-commutative crepant resolutions and their mutations gives a new alternative proof for their result.

\begin{cor}[\cite{C12, ADM15}] \label{cor CADM}
We have a functor isomorphism
\[ \KN'_{-k} \circ \KN_{N+k} \simeq P_k \]
for all $k \in \mathbb{Z}$.
\end{cor}

\begin{proof}
Let us consider the next diagram
\[ \begin{tikzcd}
\D(Y) \arrow{d}{P_k} \arrow{r}{\Psi_{N+k}} & \D(\modu(\Lambda_{N+k})) \arrow{d}{\nu_{N+k}^-} & \D(Y^+) \ar[l, "\Psi^+_{-k-1}"'] \ar{ld}{\Psi^+_{-k}} \\
\D(Y) \arrow[r, "\Psi_{N+k-1}"'] & \D(\modu(\Lambda_{N+k-1})). &
\end{tikzcd} \]
Since $(\Psi^+_{-k-1})^{-1} \circ \Psi_{N+k} \simeq \KN_{N+k}$ and $(\Psi_{N+k-1})^{-1} \circ \Psi^+_{-k} \simeq \KN'_{-k}$ by Theorem \ref{thm kn=nkn},
we have $\KN'_{-k} \circ \KN_{N+k} \simeq P_k$.
\end{proof}

We note that, in order to prove this corollary, Cautis used an elaborate framework ``categorical $\mathfrak{sl}_2$-action" that is established by Cautis, Kamnitzer, and Licata \cite{CKL10, CKL13}.
Addington, Donovan, and Meachan provided two different proofs.
The first one uses a technique of semi-orthogonal decomposition, and the second one uses the variation of GIT quotients and ``window shifts".

%%%%%%%%%%%%%%%%%%%%%%%%%%%%%%%%%%%%%%%%%%%%%%%%%%%%%%%%%%%%%%%%%%%%%


\begin{thebibliography}{99}
\bibitem[ADM15]{ADM15} N. Addington, W. Donovan, C. Meachan, \textit{Mukai flops and P-twists}, preprint (2015), \url{https://arxiv.org/abs/1507.02595}.

\bibitem[Bei79]{Bei79} A. Beilinson, \textit{Coherent sheaves on $\mathbb{P}^n$ and problems in linear algebra}, Funct. Anal. Appl. \textbf{12}(3) (1978), 68--69.

\bibitem[Bea00]{Be00} A. Beauville, \textit{Symplectic singularities}, Invent. Math. \textbf{139}(3) (2000), 541--549.

\bibitem[Boc12]{Boc12} R. Bocklandt, \textit{Generating toric noncommutative crepant resolutions}, J. Algebra, \textbf{364} (2012), 119--147.

\bibitem[Bourbaki]{Bourbaki} N. Bourbaki, \textit{Commutative algebra. {C}hapters 1--7}, Elements of Mathematics (Berlin), Translated from the French, Reprint of the 1989 English translation, Springer-Verlag, Berlin, 1998.

\bibitem[BH93]{BH93} W. Bruns, J. Herzog, \textit{Cohen-{M}acaulay rings}, Cambridge Studies in Advanced Mathematics 39, Cambridge University Press, Cambridge, 1993.

\bibitem[BLV10]{BLV10} R-o. Buchweitz, G. J. Leuschke, M. Van den Bergh, \textit{Non-commutative desingularization of determinantal varieties {I}}, Invent. Math., \textbf{182}(1) (2010), 47--115.

\bibitem[Ca12]{C12} S. Cautis, \textit{Flops and about: a guide}, Derived categories in algebraic geometry, 61--101, EMS Ser. Congr. Rep., Eur. Math. Soc., Z\"{u}rich, 2012.

\bibitem[CKL10]{CKL10} S. Cautis, J. Kamnitzer, A. Licata, \textit{Coherent sheaves and categorical {$\mathfrak{sl}_2$} actions}, Duke Math. J., \textbf{154}(1) (2010), 135--179.

\bibitem[CKL13]{CKL13} S. Cautis, J. Kamnitzer, A. Licata, \textit{Derived equivalences for cotangent bundles of {G}rassmannians via categorical {$\mathfrak{sl}_2$} actions}, J. Reine Angew. Math., \textbf{613} (2013), 53--99.

\bibitem[Da10]{Da10}, H. Dao, \textit{Remarks on non-commutative crepant resolutions of complete intersections}, Adv. Math., \textbf{224}(3) (2010), 1021--1030.

\bibitem[DW15]{DW15} W. Donovan, M. Wemyss, \textit{Twists and braids for general 3-fold flops}, preprint (2015), \url{https://arxiv.org/abs/1504.05320}.

\bibitem[DW16]{DW16} W. Donovan, M. Wemyss, \textit{Noncommutative deformations and flops}, Duke Math. J., \textbf{165}(8) (2016), 1397--1474.

\bibitem[Gi09]{Gi09} V. Ginzburg, \textit{Lectures on Nakajima's quiver varieties}, \url{https://arxiv.org/abs/0905.0686}.

%\bibitem[H17b]{H17} W. Hara, \textit{On derived equivalence for Abuaf flop via Non-commutative crepant resolution and Iyama-Wemyss mutation}, in preparation.

\bibitem[HN17]{HN17} A. Higashitani, Y. Nakajima, \textit{Conic divisorial ideals of Hibi rings and their applications to non-commutative crepant resolutions}, preprint (2017), \url{https://arxiv.org/abs/1702.07058}.

\bibitem[HT06]{HT06} D. Huybrechts, R. Thomas, \textit{{$\Bbb P$}-objects and autoequivalences of derived categories}, Math. Res. Lett., \textbf{13}(1) (2006), 87--98.

\bibitem[HV07]{HV07} L. Hille, M. Van den Bergh, \textit{Fourier-Mukai transforms}, In: Handbook of Tilting Theory. London Math. Soc. Lecture Note Ser., vol.332, Cambridge University Press, Cambridge (2007), 147--177..

\bibitem[IR08]{IR08} O. Iyama, I. Reiten.: \textit{Fomin-Zelevinsky mutation and tilting modules over Calabi-Yau algebras}, Am. J. Math. \textbf{130}(4) (2008), 1087--1149.


\bibitem[IW14]{IW14} O. Iyama, M. Wemyss, \textit{Maximal modifications and {A}uslander-{R}eiten duality for non-isolated singularities}, Invent. Math., \textbf{197}(3) (2014), 521--586.

\bibitem[Kaw02]{Ka02} Y. Kawamata, \textit{{$D$}-equivalence and {$K$}-equivalence}, J. Differential Geom., \textbf{61}(1) (2002), 147--171.

\bibitem[Kal08]{Kal08} D. Kaledin, \textit{Derived equivalences by quantization}, Geom. Funct. Anal., \textbf{17}(6) (2008), 1968--2004.

\bibitem[KF79]{KF79} H. Kraft, C. Procesi, \textit{Closures of conjugacy class of matrices are normal}, Invent. Math. \textbf{53} (1979), 227--247.

\bibitem[Le12]{Le12} G. J. Leuschke, \textit{Non-commutative crepant resolutions: scenes from categorical geometry}, Progress in commutative algebra 1, 2012, 293--361.

\bibitem[Na03]{Na03} Y. Namikawa, \textit{Mukai flops and derived categories}, J. Reine Angew. Math., \textbf{560} (2003), 65--76.

\bibitem[Pa91]{Pa91} D. Panyushev, \textit{Rationality of singularities and the Gorenstein property for nilpotent orbits}, Funct. Appl. \textbf{25} (1991), 225--226.

%\bibitem[Ri89]{Ri89} J. Rickard, \textit{Morita theory for derived categories}, J. Lond. Math. Soc. \textbf{39}(2) (1989), 436--456

\bibitem[\v{S}V17a]{SV15} \v{S}. \v{S}penko, M. Van den Bergh, \textit{Non-commutative resolutions of quotient singularities for reductive groups}, Invent. Math. (2017), 1--65.
\bibitem[\v{S}V15]{SV15b} \v{S}. \v{S}penko, M. Van den Bergh, \textit{Comparing the commutative and non-commutative resolutions for determinantal varieties of skew symmetric and symmetric matrices}, preprint (2015), \url{https://arxiv.org/abs/1511.07290}. 

\bibitem[\v{S}V17b]{SV17} \v{S}. \v{S}penko, M. Van den Bergh, \textit{Non-commutative crepant resolutions for some toric singularities I}, preprint (2017), \url{https://arxiv.org/abs/1701.05255}. 

\bibitem[To07]{To07} Y. Toda, \textit{On certain generalization of spherical twists}, Bull. Soc. Math. France. \textbf{135}(1) (2007), 119--134.

\bibitem[TU10]{TU10} Y. Toda, H. Uehara, \textit{Tilting generators via ample line bundles}, Adv. Math., \textbf{223}(1) (2010), 1--29.

\bibitem[VdB04a]{VdB04a} M. Van den Bergh, \textit{Three-dimensional flops and noncommutative rings}, Duke Math. J., \textbf{122}(3) (2004), 423--455.

\bibitem[VdB04b]{VdB04b} M. Van den Bergh, \textit{Non-commutative crepant resolutions}, The legacy of {N}iels {H}enrik {A}bel, 749--770, Springer, Berlin, 2004.

\bibitem[We14]{We14} M. Wemyss, \textit{Flops and Clusters in the Homological Minimal Model Program}, preprint (2014), to appear in  Invent. Math., \url{https://arxiv.org/abs/1411.7189}.

\bibitem[WZ12]{WZ12} J. Weyman, G. Zhao, \textit{Noncommutative desingularization of orbit closures for some representations of $\GL_n$}, preprint (2012), \url{https://arxiv.org/abs/1204.0488}.

\end{thebibliography}
\end{document}